\documentclass[11pt]{article}

\usepackage[a4paper,margin=2.5cm]{geometry}

\usepackage{graphicx}
\usepackage{multirow}
\usepackage{amsmath,amssymb,amsfonts}
\usepackage{amsthm}
\usepackage{mathrsfs}
\usepackage{xcolor}
\usepackage{textcomp}
\usepackage{booktabs}
\usepackage{algorithm}
\usepackage{algorithmicx}
\usepackage{algpseudocode}
\usepackage{listings}
\usepackage{orcidlink}
\usepackage{mathtools}
\usepackage{makecell}
\usepackage{comment}
\usepackage{authblk}
\usepackage[numbers,sort&compress]{natbib}

\usepackage{tikz}
\usepackage{pgfplots}
\pgfplotsset{compat=1.18}
\usepgfplotslibrary{statistics}
\usepgfplotslibrary{polar}
\usetikzlibrary{positioning}
\usepackage{pgfplotstable}

\pgfplotsset{
	table/col sep=tab,
	table/skip first n=1,
}

\definecolor{viridisYellow}{RGB}{253,231,37}
\definecolor{viridisGreen}{RGB}{94,201,98}
\definecolor{viridisTeal}{RGB}{33,145,140}
\definecolor{viridisBlue}{RGB}{59,82,139}
\definecolor{viridisViolet}{RGB}{68,1,84}
\definecolor{matchingRed}{HTML}{D01F3C}
\definecolor{matchingOrange}{HTML}{FFC077}

\colorlet{colork0}{viridisBlue}
\colorlet{colork1}{viridisYellow}
\colorlet{colork2}{viridisGreen}

\colorlet{colorNN}{viridisViolet}

\pgfplotsset{
	DefaultStyle/.style={
		legend cell align={left},
		legend style={font=\footnotesize, fill opacity=1, draw opacity=1, text opacity=1, draw=white!80!black, row sep=-2pt},
		tick align=outside,
		tick pos=left,
		xmajorgrids,
		x grid style={white!69.0196078431373!black},
		xtick style={color=black},
		ymajorgrids,
		y grid style={white!69.0196078431373!black},
		ytick style={color=black},
		ticklabel style={font=\footnotesize},
		label style={font=\footnotesize},
		title style={font=\footnotesize},
		clip mode=individual,
		/tikz/every even column/.append style={column sep=0.1cm},
		colormap/viridis,
	}
}

\newcommand{\slopetriright}[5]{%
	\draw[thick] (axis cs:#3,{10^(#2)*(#3)^(#1)})
	-- (axis cs:#4,{10^(#2)*(#3)^(#1)})
	-- (axis cs:#4,{10^(#2)*(#4)^(#1)}) -- cycle;%
	\node[anchor=west, fill=white, inner sep=1pt, xshift=2pt]  at (axis cs:#4,{sqrt((10^(#2)*(#3)^(#1))*(10^(#2)*(#4)^(#1)))}) {$#5$};%
}
\newcommand{\slopetrileft}[5]{%
	\draw[thick] (axis cs:#3,{10^(#2)*(#3)^(#1)})
	-- (axis cs:#3,{10^(#2)*(#4)^(#1)})
	-- (axis cs:#4,{10^(#2)*(#4)^(#1)}) -- cycle;%
	\node[anchor=east, fill=white, inner sep=1pt, xshift=-2pt]  at (axis cs:#3,{sqrt((10^(#2)*(#3)^(#1))*(10^(#2)*(#4)^(#1)))}) {$#5$};%
}

\renewcommand{\d}[1]{\,\mathrm{d}#1}

\DeclareMathOperator*{\argmin}{arg\,min}

\newcommand{\R}{\mathbb{R}}
\newcommand{\N}{\mathbb{N}}

\newcommand{\normal}{\vec{n}}

\newcommand{\norm}[1]{\left\lVert#1\right\rVert}
\newcommand{\Span}[1]{\operatorname{span}\left(#1\right)}

\newcommand{\coeffs}{\beta}
\newcommand{\coercivity}{\alpha}
\newcommand{\continuity}{\gamma}

\newcommand{\energy}{\mathcal{J}}

\theoremstyle{plain}
\newtheorem{theorem}{Theorem}
\newtheorem{proposition}[theorem]{Proposition}
\newtheorem{lemma}[theorem]{Lemma}

\theoremstyle{remark}
\newtheorem{example}{Example}
\newtheorem{remark}{Remark}

\theoremstyle{definition}

\usepackage{hyperref}
\usepackage{cleveref}
\crefname{lemma}{lemma}{lemma}
\Crefname{lemma}{Lemma}{Lemma}

\begin{document}

\title{Kernel Methods in the Deep Ritz framework: Theory and practice\thanks{Both authors contributed equally to this work.}}

\author[1,2]{Hendrik Kleikamp\,\orcidlink{0000-0003-1264-5941}\thanks{\texttt{hendrik.kleikamp@uni-graz.at}}}
\author[3,4,5]{Tizian Wenzel\thanks{\texttt{wenzel@math.lmu.de}}}

\affil[1]{Institute for Analysis and Numerics, Mathematics M\"{u}nster, University of M\"{u}nster, Einsteinstra\ss{}e~62, 48149 M\"{u}nster, Germany}
\affil[2]{IDea\_Lab, University of Graz, Leechgasse~34, 8010 Graz, Austria}
\affil[3]{Department of Mathematics, Universit\"{a}t Hamburg, Bundesstra\ss{}e~55, 20146 Hamburg, Germany}
\affil[4]{Department of Mathematics, Ludwig Maximilian University of Munich, Theresienstra\ss{}e~39, 80333 Munich, Germany}
\affil[5]{Munich Center for Machine Learning, Munich, Germany}

\date{}

\maketitle

\begin{abstract}
In this contribution, kernel approximations are applied as ansatz functions within the Deep Ritz method.
This allows to approximate weak solutions of elliptic partial differential equations with weak enforcement of boundary conditions using Nitsche's method.
A priori error estimates are proven in different norms leveraging both standard results for weak solutions of elliptic equations and well-established convergence results for kernel methods.
This availability of a priori error estimates renders the method useful for practical purposes where theoretical guarantees are important.
The procedure is described in detail, meanwhile providing practical hints and implementation details.
By means of numerical examples, the performance of the proposed approach is evaluated and the theoretical findings are verified.
We further apply two-layered kernels in a high-dimensional test case, which goes beyond a purely linear method.
To the best of our knowledge, this paper contains the first extensive numerical studies and discussions comparing kernel interpolation, the kernel Galerkin method, the deep Ritz method with kernels and the ``classical'' deep Ritz method with neural networks.
We numerically verify advantages and disadvantages of the particular methods and show under which assumptions the individual approaches are beneficial.
\end{abstract}



\section{Introduction}
Application of different machine learning tools to approximate solutions to partial differential equations~(PDEs) are dating back to the 1990s \cite{lee1990neural,dissanayake1994neural} and have been an active direction of research especially in the past years~\cite{raissi2019physicsinformed,brunton2024promising}.
A particularly prominent example is the~\emph{deep Ritz method}, first introduced in~\cite{e2018deepritz}, which makes use of the equivalent formulation of a variational problem as an energy minimization task.
In the original deep Ritz method, deep neural networks are used as ansatz functions.
\par
In this contribution, we replace the neural networks in the deep Ritz method by kernel methods.
The linear structure of the resulting ansatz space allows for simple yet powerful a priori error estimates and convergence rates.
Such error estimates and convergence rates are only rarely available for neural networks in the deep Ritz context and require extensive analysis, see~\cite{lu2021priori,jiao2023error}.
Besides the theoretical derivations, this paper focuses in particular on extensive numerical experiments comparing several approaches in different problem settings.
\par
Meshless methods gained attention in the past decades to solve several different types of~PDEs, see for instance~\cite{patel2020meshless} and~\cite{belytschko1996meshless} for review articles on this topic.
Compared to classical mesh-based methods such as the finite element method, meshless methods do not require the (potentially complex and costly) meshing of the underlying computational domain.
In the context of kernel approximations and radial basis functions, there are various methods such as symmetric kernel collocation~\cite{franke1998solving}, Kansa method~\cite{Kansa1990b} or radial basis function finite differences~\cite{fornberg2015primer}.
Here we are concerned with the kernel Galerkin approach introduced in~\cite{wendland1999meshless}, which uses kernels as ansatz functions in the Galerkin method to solve variational formulations of elliptic~PDEs.
This approach was further developed and used for instance to solve elliptic and parabolic~PDEs on spheres in~\cite{narcowich2016novel,knemund2019highorder} or to define an energy conservative scheme for nonlinear wave equations in~\cite{sun2022kernelbased}.
However, it turns out that already for elliptic equations, the resulting stiffness matrices suffer from large condition numbers, see for instance~\cite{duan2008note}, and are costly to compute due to the required global numerical integration for every pair of ansatz functions.
\par
In this contribution we propose a possible alternative that can deal with these difficulties.
In particular we show that the error estimates derived in~\cite{wendland1999meshless} are still valid for our method.
We additionally improve the theoretical statements by taking into account the most recent results on approximation errors for kernel methods.
Towards the end of the paper we provide extensive numerical experiments in which we consider problems where the solutions possess different degrees of smoothness.
We investigate in detail the consequences of these different regularity properties when it comes to approximations of the solutions using kernel methods or neural networks.
The observed results agree with the theoretical statements we derived and the convergence rates for different solution regularities match those predicted by the theory.
Moreover, we witness the mentioned issues with large condition numbers of stiffness matrices in practice and show how our proposed approach can circumvent these difficulties.
As mentioned before, we also consider neural networks in the deep Ritz method and discuss in which cases the different techniques are preferable.
It turns out that the performance of the different methods -- neural networks and kernel methods -- depends on the smoothness of the solution, and thus, the choice of the method should take those differences into account.
To the best of the authors knowledge, such an in depth numerical study of the different approaches is not available in the literature so far.
\par
A class of approaches that shares some similarities with the methods discussed in this paper are randomized neural-network and random-feature discretizations for~PDEs:
In this setting, the nonlinear parameters of a neural-network ansatz are typically fixed randomly, while only the linear output coefficients are determined from the~PDE constraints, often through a least-squares or collocation formulation. 
This viewpoint underlies extreme-learning-machine collocation methods for elliptic~PDEs~\cite{calabro2021extreme},
local and domain-decomposition variants~\cite{dong2021local}, 
and random-feature methods that explicitly connect randomized neural networks with more classical numerical~PDE discretizations~\cite{chen2022bridging}. 
Moreover, the efficient and robust solution of the resulting linear systems has also been addressed, 
for instance through overlapping Schwarz preconditioners for randomized neural-network discretizations~\cite{shang2025overlapping}.
In contrast to the kernel-based approach presented here, theoretical results for these methods are typically difficult to obtain.
\par
In order to have a better comparison with deep methods based on neural networks,
we also investigate two-layered kernels~\cite{wenzel2024data}, 
which add a trainable matrix within the kernel function.
With this additional learnable part, the approach becomes nonlinear and the energy minimization problem cannot be replaced by an equivalent linear system anymore.
At the same time, the approximation capabilities are significantly increased compared to ``flat'' kernels, 
which appears to be preferable in a high-dimensional test case presented in our numerical results.
\par
The~\texttt{Python}-source code to reproduce the numerical results from~\Cref{sec:numerical-experiments} is provided in~\cite{sourcecode}. \newline
\par\noindent%
The paper is organized as follows: In~\Cref{sec:theoretical-background}, we provide the theoretical background on weak formulations of~PDEs required for the deep Ritz method.
We additionally recall some basic statements from numerical analysis of Galerkin methods that will be used later to prove a priori estimates in different norms.
Furthermore, we give details on kernel methods and state well-known approximation error estimates for these approaches in Sobolev norms.
Afterwards, in~\Cref{sec:application-kernels-to-deep-ritz}, the details regarding a conjunction of the deep Ritz method and kernel approximation are provided.
In~\Cref{sec:a-priori-error-estimation}, an a priori error estimate for the approach introduced in detail in~\Cref{sec:description-of-method} is proven based on the theoretical foundations from~\Cref{sec:theoretical-background}.
Furthermore, some practical aspects are considered in~\Cref{sec:practical-aspects}.
In~\Cref{sec:numerical-experiments} we show and discuss several numerical experiments that confirm the theoretical results given before and highlight the potential of the proposed methodology.
Finally, \Cref{sec:outlook-conclusion} finishes the paper with a conclusion and an outlook.

\section{Theoretical background}\label{sec:theoretical-background}
Before discussing the application of kernel methods in the deep Ritz approach, 
we will first review key definitions and notations related to the variational formulation of elliptic PDEs and the Dirichlet principle, 
followed by an introduction to kernel methods.
Here, we introduce the problem setting in an abstract and general manner.
The particular formulation employed in our numerical examples will be given in~\Cref{ex:weak-formulation} and is detailed in~\Cref{sec:numerical-experiments}.

\subsection{Variational formulation for elliptic PDEs and energy minimization}\label{sec:variational-formulation-energy-minimization}
Let~$V$ be a Hilbert space, $a\colon V\times V\to\R$ a continuous and coercive bilinear form and~$l\in V'$ a continuous linear functional on~$V$.
To introduce the required notation, we recall that a bilinear form~$a\colon V\times V\to\R$ is called coercive with coercivity constant~$\coercivity>0$ if it holds
\begin{align}\label{equ:definition-coercivity}
	a(u,u)\geq\coercivity\norm{u}_V^2
\end{align}
for all~$u\in V$. The bilinear form~$a$ is further called continuous if
\begin{align}\label{equ:definition-continuity}
	\left|a(u,v)\right|\leq\continuity\norm{u}_V\norm{v}_V
\end{align}
for all~$u,v\in V$.
We consider a general variational formulation of the form
\begin{align}\label{equ:weak-formulation}
	\text{Find }u\in V\text{ such that}\qquad a(u,v) = l(v) \qquad \text{for all }v\in V.
\end{align}
It is well-known that the weak formulation in~\eqref{equ:weak-formulation} is equivalent to the energy minimization problem
\begin{align}\label{equ:energy-functional}
	u = \argmin\limits_{v\in V} \energy(v)\qquad\text{with }\energy(v) \coloneqq \frac{1}{2}a(v,v)-l(v).
\end{align}
This energy minimization formulation is also known as \emph{Dirichlet's principle} and it is valid for general weak formulations with symmetric bilinear form.
The equivalent formulation of an elliptic~PDE in weak form as a minimization problem over the ansatz space~$V$ plays a crucial role in the deep Ritz method.

\begin{remark}[Non-symmetric problems and residual minimization]\label{rem:non-symmetric-problems}
	The energy minimization formulation as presented here is only valid for symmetric problems.
	In a more general setting, one might consider the (weak) residual defined in the dual space instead and perform minimization of the residual.
	The authors in~\cite{uriarte2023deep} propose to this end the ``Deep Double Ritz Method'' which solves the corresponding saddle-point problem by means of two deep Ritz minimizations.
	However, the general idea of considering an optimization problem instead of a weak formulation allows to apply the approach detailed in\Cref{sec:description-of-method} to a larger class of problems.
	The method is in principle applicable to objective functions that do not correspond to the Dirichlet energy in~\eqref{equ:energy-functional}.
	We here focus on energy minimization of elliptic~PDEs since those problems allow for a priori error estimation using tools from numerical analysis.
\end{remark}

We now given an example of a weak formulation for a particular class of problems that will be our guiding example in the numerical experiments in~\Cref{sec:numerical-experiments}. This example might be generalized including for instance reaction terms or different boundary conditions. However, already in this relatively simple setting of a Poisson problem we observe fundamentally different smoothness behaviors of the corresponding solutions depending on the choice of the domain and the boundary conditions, see~\Cref{sec:numerical-experiments} for details. We therefore restrict our attention to a Poisson problem with Dirichlet boundary conditions:
\begin{example}[Weak formulation of a diffusion problem with weak boundary value enforcement]\label{ex:weak-formulation}
	Let~$\Omega\subset\R^d$ for some~$d\in\N$ be a bounded domain.
	We assume that the boundary~$\partial\Omega$ is Lipschitz continuous.
	Let~$f\in L^2(\Omega)$ be the right-hand side and~$g\in H^{1/2}(\partial\Omega)$ the boundary condition.
	We consider as an example the weak formulation of the following Poisson problem
	\begin{equation*}
		\begin{aligned}
			-\Delta u &= f && \text{in }\Omega,\\
			u &= g && \text{on }\partial\Omega.
		\end{aligned}
	\end{equation*}
	By means of Nitsche's method, see~\cite{nitsche1971variationsprinzip}, one can derive a weak formulation with weak enforcement of the boundary condition for the diffusion problem stated above.
	We thus define the symmetric bilinear form~$a\colon V\times V\to\R$ with~$V\coloneqq H^1(\Omega)$ for~$u,v\in V$ as
	\begin{align}\label{equ:definition-bilinear-form}
		a(u,v) \coloneqq \int\limits_\Omega \nabla u\cdot\nabla v \d{x} - \int\limits_{\partial\Omega} (\nabla u\cdot \normal)v\d{s} - \int\limits_{\partial\Omega} (\nabla v\cdot \normal)u\d{s}+C_{\text{pen}}\int\limits_{\partial\Omega} uv\d{s}
	\end{align}
	and the continuous linear functional~$l\colon V\to\R$ for~$v\in V$ as
	\begin{align*}
		l(v) \coloneqq \int\limits_\Omega fv\d{x} - \int\limits_{\partial\Omega} (\nabla v\cdot \normal)g\d{s} + C_{\text{pen}}\int\limits_{\partial\Omega} gv\d{s},
	\end{align*}
	where~$C_\text{pen}>0$ denotes a penalty parameter used to weakly impose the boundary condition.
	
	The bilinear form defined in~\eqref{equ:definition-bilinear-form} is continuous and coercive for a suitable choice of the penalty parameter~$C_\text{pen}$ and,
	thus, the weak formulation in~\eqref{equ:weak-formulation} has a unique solution~$u\in V$ for that choice of the penalty parameter as well,
	see~\cite{benzaken2024constructing,ueda2018infsup} for more details and a thorough derivation.
	In practice, the constant~$C_\text{pen}$ is typically chosen heuristically since the required lower bounds for this constant are usually not known explicitly.
	
	We highlight at this point the differences of our weak formulation to the one considered in~\cite{e2018deepritz}: Besides the penalty term for the weak enforcement of the boundary condition also present in~\cite{e2018deepritz} for homogeneous Dirichlet boundary values, we added a consistency term to the bilinear form and the respective term that leads to a symmetric problem.
	This way, a strong solution of the Poisson problem is also a weak solution and the weak formulation is consistent with the strong formulation.
	The consistency of our formulation therefore enables a comparison to the strong solution in the numerical examples considered in~\Cref{sec:numerical-experiments}, which allows us to confirm the convergence rates predicted by the theory from~\Cref{sec:a-priori-error-estimation}.
	A similar formulation based on Nitsche's method was also considered in~\cite{ming2021deep} in the context of deep Ritz methods.
\end{example}

Later on, we will consider approximations of the weak solution~$u\in V$ in subspaces~$V_h\subset V$ of~$V$.
To analyze the error made by the approximation in the subspace~$V_h$, the following result by C\'{e}a gives an error estimate for the so-called~\emph{Galerkin projection}~$u_h\in V_h$ by means of the best-approximation error in the space~$V_h$.
Subsequently, the best-approximation error will be bounded further in the case of an approximation via kernel methods.
\begin{theorem}[{C\'{e}a's~Lemma; see~\cite{cea1964approximation}}]\label{thm:ceas-lemma}
	Let~$V_h\subset V$ be a closed subspace of~$V$.
	Further, let~$a\colon V\times V\to\R$ be a symmetric, continuous and coercive bilinear form and~$l\in V'$ a continuous linear functional.
	We denote by~$u\in V$ the weak solution in~$V$ such that it holds~$a(u,v)=l(v)$ for all~$v\in V$, and by~$u_h\in V_h$ the weak solution in~$V_h$, i.e. it holds~$a(u_h,v_h)=l(v_h)$ for all~$v_h\in V_h$.
	Then the quasi-best approximation estimate
	\begin{align*}
		\norm{u-u_h}_{V} \leq \sqrt{\frac{\continuity}{\coercivity}}\inf_{v_h\in V_h}\norm{u-v_h}_{V}
	\end{align*}
	holds, where~$\continuity>0$ is the continuity constant, see~\eqref{equ:definition-continuity}, and~$\coercivity>0$ denotes the coercivity constant of the bilinear form~$a$ with respect to the norm on~$V$, see~\eqref{equ:definition-coercivity}.
\end{theorem}
Below we will prove an a priori error estimate for the deep Ritz method using kernel approximations with respect to the~$H^1$-norm. 
This result naturally gives rise to an estimate in the~$L^2$-norm since the~$L^2$-norm is bounded from above by the~$H^1$-norm. However, the estimate with respect to the~$L^2$-norm can be improved further using the well-known Aubin-Nitsche Lemma, which will be recalled in the following in a general and abstract setting.
\begin{theorem}[{Aubin-Nitsche Lemma; see~\cite{ciarlet2002finite}}]\label{thm:aubin-nitsche-lemma}
	Let~$V_h\subset V$ be a closed subspace of~$V$ and let~$W$ be a Hilbert space such that the embedding~$V\hookrightarrow W$ is continuous. Further, let the assumptions from~\Cref{thm:ceas-lemma} be fulfilled. Then it holds the estimate
	\begin{align*}
		\norm{u-u_h}_W \leq \continuity\norm{u-u_h}_V\sup\limits_{0\neq w\in W}\frac{\inf\limits_{v_h\in V_h}\norm{v_w-v_h}_V}{\norm{w}_W},
	\end{align*}
	where~$v_w\in V$ is for a given~$w\in W$ defined as the solution of the dual problem
	\begin{align}\label{equ:dual-problem}
		a(\varphi,v_w) = (w,\varphi)_W\qquad\text{for all }\varphi\in V.
	\end{align}
\end{theorem}

\subsection{Kernel methods and approximation error estimates}
\label{sec:kernel-methods}

Kernel methods constitute a popular machine learning method and are well-known for a comprehensive theoretical foundation, 
see for instance~\cite{wendland2004scattered} for a detailed overview. 
These methods build in our case around the notion of strictly positive-definite kernels, i.e.~symmetric functions~$k\colon\R^d\times\R^d\to\R$, 
for which the kernel matrix~$\left[k(x_i,x_j)\right]_{i,j=1}^{n}\in\R^{n\times n}$ is positive-definite for all sets of pairwise distinct points~$x_1,\dots,x_n\in\R^d$ and all~$n\in\N$.
\par
From now on, let~$k$ be a radial basis function kernel, 
i.e.~there exists a radial basis function~$\Phi\colon\R^d\to\R$ such that~$k(x,y)=\Phi(x-y)=\varphi(\norm{x-y}_2)$ with~$\varphi\colon\R_{\geq0}\to\R$ for all~$x,y\in\R^d$. 
We assume that for the Fourier transform~$\hat{\Phi}\colon\R^d\to\R$ of~$\Phi$ the algebraic decay property
\begin{align}\label{equ:Fourier-transform-decay}
	c_1(1+\norm{\omega}_2^2)^{-\tau} \leq \hat{\Phi}(\omega) \leq c_2(1+\norm{\omega}_2^2)^{-\tau} \qquad \text{for all }\omega\in\R^d
\end{align}
holds, where~$c_2 > c_1 > 0$ are some constants and~$\tau>d/2$ is the order of the decay. 
\par
For any strictly positive definite kernel~$k$, there exists a unique reproducing kernel Hilbert space (RKHS)~$\mathcal{H}_k$ of functions, the so-called native space.
The~RKHS is defined by the reproducing property, i.e.~that it holds~$\langle f,k(\,\cdot\,,x)\rangle_{\mathcal{H}_k}=f(x)$ for all~$f\in\mathcal{H}_k$ and all~$x\in\R^d$, where~$\langle\,\cdot\,,\,\cdot\,\rangle_{\mathcal{H}_k}\colon\mathcal{H}_k\times\mathcal{H}_k\to\R$ denotes the inner product on the Hilbert space~$\mathcal{H}_k$.
Under the assumption~\eqref{equ:Fourier-transform-decay},
the RKHS of~$k$ over a Lipschitz domain~$\Omega$ can be shown to be norm-equivalent to the Sobolev space~$H^\tau(\Omega)$.
We will assume~$\Omega$ to be a Lipschitz domain in the following.
\par
To derive error estimates for kernel methods, one usually needs assumptions on the regularity of the domain~$\Omega\subset\R^d$.
A~typical requirement is that the domain~$\Omega$ satisfies an interior cone condition, i.e.~there exists a maximum radius~$R_0>0$ and a fixed opening angle~$0<\psi\leq\frac{\pi}{2}$ such that for all~$x\in\Omega$ there exists a cone~$C_{\xi(x),R,\psi}\subset\R^d$ with axis~$\xi(x)\in\R^d$, $\norm{\xi(x)}_2=1$ and~radius~$R\leq R_0$, defined as
\begin{align*}
	C_{\xi(x),R,\psi}=\left\{y\in\R^d:\norm{y}_2\leq R\quad\text{and}\quad\langle\xi(x),y\rangle>\norm{y}_2\cos(\psi)\right\},
\end{align*}
such that~$x+C_{\xi(x),R,\psi}\subseteq\Omega$, see~\cite[Section~2.2]{narcowich2004sobolev} for details.
In other words, there is a cone with fixed radius and opening angle that can be placed at every point of the domain~$\Omega$ and rotated in such a way that it is contained in~$\Omega$.
The boundary of~$\Omega$ can therefore not contain any cusps.
\par
In particular, kernels can be used for interpolation and regression tasks: Given a finite set~$X=\{x_1,\dots,x_n\}\subset\R^d$ with corresponding target values~$y_1,\dots,y_n\in\R$, the respective approximation task minimizing the mean-square error reads as
\begin{align*}
	s_n=\argmin_{f\in\mathcal{H}_k} \frac{1}{n}\sum\limits_{i=1}^{n}\lvert y_i-f(x_i)\rvert^2+\lambda\norm{f}_{\mathcal{H}_k}^2,
\end{align*}
where~$\lambda\geq 0$ denotes a regularization parameter that aims to avoid overfitting of the target data and~$\norm{\,\cdot\,}_{\mathcal{H}_k}$ is the norm on the Hilbert space~$\mathcal{H}_k$ induced by the inner product~$\langle\,\cdot\,,\,\cdot\,\rangle_{\mathcal{H}_k}$.
A well-known representer theorem, see for instance~\cite{kimeldorf1970correspondence}, states that solutions to this optimization problem are of the form
\begin{align}\label{equ:definition-kernel-ansatz-function}
	s_n = \sum_{j=1}^n \coeffs_j k(\,\cdot\,,x_j),
\end{align}
where the coefficients~$\coeffs_1,\dots,\coeffs_n\in\R$ define the kernel approximant.
When restricting to a pure interpolation task, i.e.~$\lambda=0$, the ansatz function~$s_n$ from~\eqref{equ:definition-kernel-ansatz-function} together with the interpolation conditions~$s_n(x_i) = y_i$ for all~$i=1,\dots,n$ yield the linear system of equations
\begin{align*}
	K \boldsymbol{\coeffs} = y
\end{align*}
for the vector of coefficients~$\boldsymbol{\coeffs}=[\coeffs_i]_{i=1}^n\in\R^n$, where~$K\in\R^{n\times n}$ denotes the kernel matrix with entries~$K_{ij} = k(x_i, x_j)$ for~$1 \leq i, j \leq n$ and~$y = [y_1,\dots, y_n]^\top\in\R^n$.
Due to the assumption of~$k$ being strictly positive definite, this linear system of equations admits a unique solution since the kernel matrix~$K$ is positive definite.
Kernel based approximation thus takes place in the (finite-dimensional) ansatz space
\begin{align}\label{equ:definition-span-kernel-functions}
	V_{X,k} \coloneqq \Span{k(\,\cdot\,,x_i):x_i\in X} = \Span{\Phi(\,\cdot\,-x_i):x_i\in X}.
\end{align}
In the case of approximation of a function~$f\colon\Omega\to\R$, i.e.~the case~$y_1 = f(x_1),\dots, y_n = f(x_n)$, it is possible to quantify the approximation error between the kernel approximant~$s_n$ and the function~$f$.
To this end, one usually makes use of the mesh norm~$h$ (also called fill distance in the context of scattered data) of the finite set~$X=\{x_1,\dots,x_n\}\subset\Omega$ in~$\Omega\subset\R^d$, 
which is defined as
\begin{align*}
	h \coloneqq h_{X, \Omega} \coloneqq \sup_{x\in\Omega}\min_{x_i\in X} \norm{x-x_i}_2.
\end{align*}
We here use the term ``mesh norm'' instead of ``fill distance'' since in our unsupervised learning setting we can place the centers on a uniform mesh in order to reduce the mesh norm.
In the context of supervised learning tasks, the data that is to be approximated is usually given and therefore, according to the representer theorem, also the positions of the centers are determined by the data.
\par
With this notation, we state the following approximation error estimate for functions in a Sobolev space by functions from the space~$V_{X,k}$ that was proven in~\cite{narcowich2004sobolev}:
\begin{theorem}[{Approximation error estimate for kernel methods in Sobolev norms;
		see~\cite[Theorem~3.8]{narcowich2004sobolev}}]
	\label{thm:approximation-error-estimate}
	Let~$k$ be a radial basis function kernel with radial basis function~$\Phi$ satisfying~\eqref{equ:Fourier-transform-decay} with~$\tau\geq t>d/2$, let~$\Omega$ satisfy an interior cone condition, and assume that~$X=\{x_1,\dots,x_n\}\subset\Omega$ has mesh norm~$h$ that fulfills (the technical) condition~(9) in~\cite{narcowich2004sobolev}.
	If~$u\in H^t(\Omega)$, then there exists a function~$v\in V_{X,k}=\Span{\Phi(\cdot-x_i):x_i\in X}$ such that for every real~$0\leq r\leq t$ it holds
	\begin{align*}
		\norm{u-v}_{H^r(\Omega)} \leq Ch^{t-r}\norm{u}_{H^{t}(\Omega)},
	\end{align*}
	where~$C$ is a constant independent of~$u$ and~$h$.
\end{theorem}

\Cref{thm:approximation-error-estimate} provides a bound on the error when approximating functions in a Sobolev space by functions from the subspace~$V_{X,k}$.
In the context of kernel methods one is usually interested in interpolation error estimates where the function~$v\in V_{X,k}$ is chosen as the interpolant of the function~$u$ in the centers.
However, in the following section we consider an unsupervised learning approach where the approximation~$v\in V_{X,k}$ is not an interpolant but a quasi-best approximation of~$u$.
For the a priori error estimates we therefore need a bound for the best approximation error by functions in~$V_{X,k}$ which is provided by~\Cref{thm:approximation-error-estimate}.
Certainly, interpolation error estimates, see for instance~\cite{narcowich2006sobolev}, also provide an upper bound for the best approximation error, but these estimates tend to be less effective when the points in~$X$ are closely spaced.
We thus resort to an approximation error estimate which fits to the proof of~\Cref{thm:a-priori-error-estimate-h1} based on~C\'{e}a's~Lemma.
We also remark that in the context of finite element methods, error estimates for interpolations of functions from~Sobolev spaces by functions from finite element spaces are used to prove a priori error estimates.
\par
Finally, we briefly review \textit{two-layered kernels} as introduced in~\cite{wenzel2024data}:
The key idea is to enhance the radial kernel by incorporating an optimizable (square) matrix~$B \in \R^{d \times d}$,
as~$k_B(x, y) = \varphi(\Vert B(x - y) \Vert)$.
This compositional structure can be viewed mathematically as a two-layered kernel, 
thus motivating its name.
The matrix~$B$ can be adapted (e.g.\ via an optimization) to the problem at hand,
thereby increasing the performance of the kernel model.
Typical cases include situations with strong anisotropy in the data or function to be approximated.
Due to the mathematical simplicity of the two-layered approach,
many asymptotic properties of kernel models can be transferred to such two-layered kernel models.

\subsection{Neural networks and the deep Ritz method}\label{sec:nns-and-deep-ritz}
We now briefly summarize the main aspects of the deep Ritz method using neural networks to approximate the weak solution of the underlying~PDE, as originally introduced in~\cite{e2018deepritz}.
To this end, let us start with the definition of fully-connected feed-forward neural networks following~\cite{petersen2018optimal}.
For simplicity, we restrict ourselves to neural networks with scalar output, which fits to our application in the context of scalar~PDEs.
\par
A neural network is a function~$\Psi_\theta\colon\R^d\to\R$, parametrized by weights and biases denoted as~$\theta$, of the form
\begin{align*}
	\Psi_\theta(x) &\coloneqq W_Lr^{L-1}(x)+b_L,\\
	r^i(x) &\coloneqq \rho(W_ir^{i-1}(x)+b_i),\qquad\text{for }i=1,\ldots,L-1,\\
	r^0(x) &\coloneqq x.
\end{align*}
Here, we denote by~$W_1\in\R^{n_1\times d},W_2\in\R^{n_2\times n_1}\ldots,W_L\in\R^{1\times n_{L-1}}$ the weights and by~$b_1\in\R^{n_1},b_2\in\R^{n_2},\ldots,b_L\in\R$ the biases and define~$\theta=[W_1,\ldots,W_L,b_1,\ldots,b_L]$ collecting all tunable parameters of the network.
The number of layers is given as~$L\in\N$ and the so-called activation function~$\rho$ acts as a component-wise nonlinearity on the inputs.
\par
In the original deep Ritz method, feed-forward neural networks are used as trial functions instead of classical ansatz spaces such as finite elements.
The energy~$\energy(v)$ in~\eqref{equ:energy-functional} is minimized for~$v=\Psi_\theta$ as a function of the weights and biases~$\theta$.
To be more precise, the optimal network parameters~$\theta^*$ are determined as
\begin{align*}
	\theta^* = \argmin_{\theta} \energy\left(\Psi_\theta\right).
\end{align*}
The deep Ritz solution is then given as the neural network~$\Psi_{\theta^*}$.
Evaluating the energy functional for a neural network requires approximation of the involved integrals using, for instance, Monte Carlo methods.
For the minimization, standard optimization algorithms from the machine learning literature, such as stochastic gradient descent, can be employed.
The gradient of the energy functional with respect to the neural network parameters can be computed using automatic differentiation~\cite{baydin2017automatic}.
The method has been investigated regarding its convergence in~\cite{mueller2020deep} using the notion of~$\Gamma$-convergence.
However, a priori error estimates for the approach are typically intricate to derive, see~\cite{lu2021priori,jiao2023error}.

\section{Application of kernel approximations in the deep Ritz method}\label{sec:application-kernels-to-deep-ritz}
As outlined in the previous section, the main idea of the deep Ritz method, introduced in~\cite{e2018deepritz}, is to apply deep neural networks as ansatz functions in the energy minimization problem for the functional~$\energy$ from~\eqref{equ:energy-functional}.
\par
Instead of using deep neural networks, we consider kernel approximations as introduced in~\Cref{sec:kernel-methods} as ansatz functions.
We would like to emphasize at this point again that for a fixed kernel~$k$ and fixed centers~$X=\{x_1,\dots,x_n\}$, the set of kernel functions for different coefficients forms the linear space~$V_{X,k}$ as introduced in~\eqref{equ:definition-span-kernel-functions}.
This is not the case for the set of neural networks with different weights when considering a fixed neural network architecture, which does not form a linear space.
Hence, for instance C\'{e}a's~Lemma, see~\Cref{thm:ceas-lemma}, cannot be applied to the set of neural networks.
\par
In the next three subsections, we focus on the deep Ritz method for ``flat'' kernels, i.e.~considering the linear space~$V_{X,k}$.
We extend the approach to two-layered kernels in~\Cref{sec:deep-ritz-two-layered}.
\par
The application of kernel methods instead of neural networks in the deep Ritz approach is motivated by the available theoretical results that will be established in~\Cref{sec:a-priori-error-estimation}.
For sufficiently regular solutions we further expect satisfactory numerical results, see~\Cref{sec:numerical-experiments}.
Furthermore, the formulation as an energy minimization problem instead of a linear system of equations (resulting from the weak equation in~\eqref{equ:weak-formulation}) has several advantages when using kernel approximants as ansatz functions.
First of all, the entries of the stiffness matrix~$A\in\R^{n\times n}$ using the space~$V_{X,k}$ as ansatz and test space are given for~$i,j=1,\dots,n$ as
\begin{align}\label{equ:definition-stiffness-matrix}
	A_{ij} = a\big(k(\,\cdot\,,x_j),k(\,\cdot\,,x_i)\big).
\end{align}
This matrix could be subsequently used to solve~\eqref{equ:weak-formulation} in the subspace~$V_{X,k}$ via 
\begin{align}\label{equ:system-stiffness-matrix}
	A \boldsymbol{\coeffs} = \ell
\end{align}
to determine the approximant~$s_n = \sum_{j=1}^n \coeffs_j k(\,\cdot\,, x_j)$ with coefficients~$\boldsymbol{\coeffs} = [\coeffs_i]_{i=1}^n \in \R^n$ and right-hand side vector~$\ell = [\ell(k(\,\cdot\,, x_i))]_{i=1}^n \in \R^n$.
\par
We observe that this matrix has a similar structure than the kernel matrix~$K$ in the sense that both can be seen as Gramian matrices, i.e.~as matrices containing pairwise evaluations of inner products or bilinear forms.
It has been discussed theoretically in~\cite{duan2008note}, 
and we also observe this issue numerically below, that the stiffness matrix~$A$ of the weak formulation might be ill-conditioned when using kernels as ansatz and test functions. 
This results in difficulties when solving linear systems with the stiffness matrix.
Additionally, computing the entries of the stiffness matrix is challenging due to the required accurate numerical quadrature. 
In finite element methods, the ansatz and test functions have support only on a small number of elements and are usually polynomials of moderate degree, 
which enables an exact computation of the required integrals (as long as the data functions can also be integrated exactly). 
Another drawback of the linear system is that the stiffness matrix might be dense in the kernel-based approach. 
Hence, solving the linear system is not only difficult due to the large condition number of the system matrix but is also costly because the stiffness matrix is not sparse. 
This is again in contrast to finite element methods where the ansatz functions have compact support and the resulting system can be solved very efficiently by means of iterative methods. 
However, using the formulation of the weak problem via the energy minimization, we can circumvent the issues described before. 
We will discuss these advantages in detail based on numerical examples in~\Cref{sec:numerical-experiments}.

\subsection{Description of the method}\label{sec:description-of-method}
In the following, we describe the deep Ritz approach using kernel methods in more detail.
To this end, assume that we are given a fixed set of points~$X=\{x_1,\dots,x_n\}\subset\Omega$ and a kernel~$k$.
We denote by~$u_h\in V_h=V_{X,k}\subset H^1(\Omega)$ the weak solution of the variational problem~\eqref{equ:weak-formulation} considered over the space~$V_{X,k}$, which will be computed by minimizing the energy functional from~\eqref{equ:energy-functional}, i.e.~we have
\begin{align}\label{equ:energy-minimization-kernel-approximation}
	u_h = \argmin_{v_h\in V_{X,k}} \energy(v_h).
\end{align}
The subscript~$h$ refers to the mesh norm that will also appear in the a priori error estimate and can be seen as the discretization parameter of the space~$V_{X,k}$ with respect to the space~$V$. 
In that sense, the mesh norm has a similar meaning as the grid size in finite element methods.
\par
As before, for a fixed set~$X=\{x_1,\dots,x_n\}$, we consider ansatz functions~$\psi_h[\boldsymbol{\coeffs}]\in V_{X,k}$ of the form
\begin{align}\label{equ:linear-combination-kernel}
	\psi_h[\boldsymbol{\coeffs}](x) = \sum\limits_{i=1}^{n}\coeffs_i k(x,x_i)
\end{align}
for~$x\in\Omega$ with coefficient vector~$\boldsymbol{\coeffs}=\left[\coeffs_i\right]_{i=1}^{n}\in\R^n$. 
The optimization problem in~\eqref{equ:energy-minimization-kernel-approximation} can thus be rewritten as
\begin{align}\label{equ:optimization-problem-coefficients}
	\boldsymbol{\coeffs}^* = \argmin_{\boldsymbol{\coeffs}\in\R^n} \energy\big(\psi_h[\boldsymbol{\coeffs}]\big)
\end{align}
and~$u_h$ is then given as~$u_h=\psi_h[\boldsymbol{\coeffs}^*]$.
In other words, given a kernel and a set of centers, we search for coefficients such that the linear combination of the kernel centered at the given center points (as given in~\eqref{equ:linear-combination-kernel}) minimizes the energy functional.
Due to the equivalence of the variational formulation and the energy minimization problem, solving~\eqref{equ:optimization-problem-coefficients} exactly results in the weak solution of the variational problem~\eqref{equ:weak-formulation} over the space~$V_{X,k}$, i.e.~it holds
\begin{align*}
	a(u_h,v_h) = l(v_h) \qquad \text{for all }v_h\in V_{X,k}.
\end{align*}
After a suitable discretization of the energy functional~$\energy$,
we can solve the optimization problem in~\eqref{equ:optimization-problem-coefficients} for the optimal coefficients using standard gradient-based methods from the machine learning literature, 
see~\Cref{sec:practical-aspects} for details on the practical implementation of the method.
\par
We emphasize at this point that we consider weak boundary values within the weak formulation in~\eqref{equ:weak-formulation}.
This is similar to the original deep Ritz method using neural networks as ansatz functions.
Applying kernel approximants instead of neural networks results in the same difficulty that strong boundary values are virtually impossible to prescribe for these kinds of functions.
We thus proceed with a similar weak boundary value enforcement as performed in~\cite{e2018deepritz}.
\par
Due to the linear structure of the ansatz space~$V_{X,k}$ and the thorough approximation error analysis available for kernel methods, it is possible to derive a priori error estimates for the solution~$u_h\in V_{X,k}$ of the energy minimization problem.
Such error estimates with respect to different norms are presented in the next section.

\subsection{A priori error estimation}\label{sec:a-priori-error-estimation}
Using the theoretical results on weak formulations and on approximation errors for kernel methods, we now formulate an a priori error estimate with respect to the~$H^1$-norm for the error of the kernel based solution~$u_h\in V_{X,k}$ from~\Cref{sec:application-kernels-to-deep-ritz}.
Having the theoretical tools from the previous sections at hand, the proof of the error estimate is relatively simple and follows the standard procedure for deriving a priori error estimates for weak solutions of elliptic~PDEs.
The error estimate is also similar to the one for Lagrange finite elements of order~$m-1$ where~$h$ denotes the grid size of the triangulation.
We highlight at this point as well that the result holds similarly for arbitrary weak formulations which are equivalent to energy minimization problems.
\begin{theorem}[A priori error estimate in the~$H^1$-norm]\label{thm:a-priori-error-estimate-h1}
	Let~$\Omega\subset\R^d$ be a bounded domain and assume that~$\Omega$ satisfies an interior cone condition. Let~$V=H^1(\Omega)$, let~$a\colon V\times V\to\R$ be a symmetric, continuous and coercive bilinear form, and let~$l\in V'$ be a continuous linear functional on~$V$. Furthermore, let~$u\in V$ be the weak solution of the variational problem~\eqref{equ:weak-formulation}. Assume that the positive definite radial basis function kernel~$k\colon\Omega\times\Omega\to\R$ has an algebraically decaying Fourier transform of order~$\tau\geq m>d/2$ for some~$m\in\R$, i.e.~it holds~\eqref{equ:Fourier-transform-decay} with~$\tau$ given as stated before. Let~$X=\{x_1,\dots,x_n\}\subset\Omega$ have mesh norm~$h>0$ such that condition~(9) in~\cite{narcowich2004sobolev} is fulfilled. Assuming~$u\in H^m(\Omega)$, we have that there exists a constant~$c_{H^1}>0$ independent of~$h$ and~$l$ such that for the kernel approximation~$u_h\in V_h=V_{X,k}\subset V$, which solves the energy minimization problem in~\cref{equ:energy-minimization-kernel-approximation}, it holds
	\begin{align*}
		\norm{u-u_h}_{H^1(\Omega)} \leq c_{H^1}h^{m-1}\norm{u}_{H^m(\Omega)}.
	\end{align*}
\end{theorem}
\begin{proof}
	We first note that since~$u_h\in V_{X,k}$ solves the energy minimization problem, it is also the unique weak solution of the variational problem~$a(u_h,v_h)=l(v_h)$ for all~$v_h\in V_{X,k}$ according to Dirichlet's~principle, see~\Cref{sec:variational-formulation-energy-minimization}.
	Therefore, applying first~C\'{e}a's~Lemma with~$V_h=V_{X,k}$, see~\Cref{thm:ceas-lemma}, and then the approximation error estimate from~\Cref{thm:approximation-error-estimate} with~$\tau=t=m$ and~$r=1$, we obtain
	\begin{align*}
		\norm{u-u_h}_{H^1(\Omega)} &\leq \sqrt{\frac{\continuity}{\coercivity}}\inf_{v_h\in V_{X,k}}\norm{u-v_h}_{H^1(\Omega)} \\
		&\leq C\sqrt{\frac{\continuity}{\coercivity}}h^{m-1}\norm{u}_{H^m(\Omega)} \\
		&\leq c_{H^1}h^{m-1}\norm{u}_{H^m(\Omega)}
	\end{align*}
	with~$c_{H^1}\coloneqq C\sqrt{\frac{\continuity}{\coercivity}}$, which proves the theorem.
\end{proof}

\begin{remark}[Differences to the a priori error estimate in~\cite{wendland1999meshless}]
	In Corollary~5.4 in~\cite{wendland1999meshless}, it was necessary to assume~$m>d/2+1$.
	Since this is not required in the approximation error estimate in~\Cref{thm:approximation-error-estimate}, we can weaken the assumption to~$m>d/2$ in the a priori error estimate as well.
	\Cref{thm:approximation-error-estimate} therefore constitutes an updated version of the result in~\cite{wendland1999meshless} using the latest approximation error estimates.
	We emphasize at this point that~\Cref{thm:a-priori-error-estimate-h1} in particular implies that the radial basis function kernel~$k$ is allowed to have a higher regularity (i.e.~$\tau>m$) than the solution~$u\in H^m(\Omega)$ itself.
	In contrast, standard interpolation error estimates for Lagrange finite elements typically require that the solution has a higher regularity than the polynomial degree of the ansatz functions.
\end{remark}

We now also derive an error estimate with respect to the~$L^2$-norm that obtains the optimal order of convergence (i.e.~order~$m$ instead of order~$m-1$ as for the~$H^1$-norm). To that end, we combine the previous~$H^1$-estimate with the Aubin-Nitsche~Lemma from~\Cref{thm:aubin-nitsche-lemma} and obtain the following result:
\begin{theorem}[A priori error estimate in the~$L^2$-norm]\label{thm:a-priori-error-estimate-l2}
	Let the assumptions from~\Cref{thm:a-priori-error-estimate-h1} be fulfilled. Assume further that it holds~$m\geq 2$ and that for every~$w\in L^2(\Omega)$ the solution of the dual problem~\eqref{equ:dual-problem} satisfies~$v_w\in H^2(\Omega)$ and~$\norm{v_w}_{H^2(\Omega)}\leq \bar{C}\norm{w}_{L^2(\Omega)}$ for some fixed constant~$\bar{C}>0$. Then there exists a constant~$c_{L^2}>0$ such that the estimate
	\begin{align*}
		\norm{u-u_h}_{L^2(\Omega)} \leq c_{L^2}h^m\norm{u}_{H^m(\Omega)}
	\end{align*}
	holds.
\end{theorem}
\begin{proof}
	We apply the Aubin-Nitsche~Lemma from~\Cref{thm:aubin-nitsche-lemma} with~$V=H^1(\Omega)$ and~$W=L^2(\Omega)$. We can hence estimate
	\begin{align*}
		\norm{u-u_h}_{L^2(\Omega)} &\leq \continuity\norm{u-u_h}_{H^1(\Omega)}\sup\limits_{0\neq w\in L^2(\Omega)}\frac{\inf\limits_{v_h\in V_h}\norm{v_w-v_h}_{H^1(\Omega)}}{\norm{w}_{L^2(\Omega)}} \\
		&\leq \continuity c_{H^1}h^{m-1}\norm{u}_{H^m(\Omega)}\sup\limits_{0\neq w\in L^2(\Omega)}\frac{Ch\norm{v_w}_{H^2(\Omega)}}{\norm{w}_{L^2(\Omega)}} \\
		&\leq \continuity c_{H^1}h^{m-1}\norm{u}_{H^m(\Omega)}\sup\limits_{0\neq w\in L^2(\Omega)}\frac{C\bar{C}h\norm{w}_{L^2(\Omega)}}{\norm{w}_{L^2(\Omega)}} \\
		&= c_{L^2}h^m\norm{u}_{H^m(\Omega)}
	\end{align*}
	with~$c_{L^2}\coloneqq\continuity c_{H^1}C\bar{C}$, where we made use of the result in the~$H^1$-norm from~\Cref{thm:a-priori-error-estimate-h1} and the approximation error from~\Cref{thm:approximation-error-estimate} to bound the approximation error of the dual solution~$v_w$.
\end{proof}
The theoretical results obtained in this section provide estimates of the approximation error achieved by the kernel approximation of the weak solution in terms of the mesh norm~$h$. 
We observe that, compared to finite element methods, a relatively high degree of regularity of the weak solution~$u$ is required for the results to hold. 
The necessary regularity is closely related to the conditions under which approximation error estimates for kernel approximants are valid. 
Further, we see that the convergence rate of the approximation via kernel approaches is restricted by the regularity of the kernel or the solution.
Indeed, if the kernel has a higher regularity than the solution, i.e.~$\tau > m$, the convergence rate is limited by the regularity of the solution.
This case is often referred to as ``escaping the native space''.
Conversely, if the kernel is less smooth than the solution, i.e.~$\tau < m$, 
the order of convergence is usually limited by the regularity of the kernel.
This behavior will also be discussed in the numerical experiments shown below. 
We further emphasize that the convergence rates are similar as for kernel interpolation of the true solution since the minimizer of the energy functional is a quasi-best approximation of the weak solution according to~C\'{e}a's~Lemma. 
We hence obtain a solution whose performance in terms of approximation error with respect to the weak solution is similar to the kernel interpolant without requiring access to the function that is to be interpolated.
\par
The a priori error estimates imply in particular that, from a theoretical perspective, reducing the mesh norm by increasing the number of centers in a suitable way results in smaller approximation errors.
Moreover, depending on the smoothness of the weak solution, employing a kernel with higher regularity leads to improved convergence rates.
Such statements on how to improve the accuracy of the method are typically not available when considering neural networks as ansatz functions.
Approaches based on neural networks often lack strategies for systematically adjusting the architecture (number of layers, number of neurons, activation function, etc.) in order to obtain a certain approximation quality.
\par
We conclude the section with a remark on how to derive upper bounds independent of the norm of the weak solution.
\begin{remark}[Estimates depending only on the problem data]
	To obtain a priori bounds that are independent of the weak solution~$u\in H^m(\Omega)$ but only depend on the data of the problem, that is, on the bilinear form and the right hand side, one can use regularity results for weak formulations of boundary-value problems similar to the one assumed in~\Cref{thm:a-priori-error-estimate-l2}, see for instance~\cite[Chapter~6.3]{evans2010partial}.
\end{remark}

\subsection{Practical aspects}\label{sec:practical-aspects}
For the practical implementation of the deep Ritz method with kernel functions, one has to approximate the bilinear form~$a$ and the linear functional~$l$ occurring in the energy minimization problem.
In most applications, the variational formulation is motivated by the weak formulation of a~PDE and thus consists of integrals of the ansatz and test functions or their derivatives over the domain~$\Omega$.
These integrals can usually not be evaluated exactly, in particular not for arbitrary kernel functions as ansatz functions.
One therefore deals with a bilinear form~$a_{h_{\mathrm{int}}}\colon V_h\times V_h\to\R$ and a linear functional~$l_{h_{\mathrm{int}}}\in V_h'$ that involve a numerical quadrature and constitute approximations to~$a$ and~$l$, respectively.
To obtain an a priori error estimate in the situation where~$u_h\in V_h$ solves the energy minimization problem for the approximate energy functional~$\energy_{h_{\mathrm{int}}}\colon V_h\to\R$, defined for~$v_h\in V_h$ as~$\energy_{h_{\mathrm{int}}}(v_h) \coloneqq \frac{1}{2}a_{h_{\mathrm{int}}}(v_h,v_h)-l_{h_{\mathrm{int}}}(v_h)$, it is possible to apply Strang's~Lemma, see~\cite{strang1972variational}, where the error introduced by the quadrature rule can be taken into account.
\par
Besides the approximation of the energy functional, the optimization problem in~\eqref{equ:optimization-problem-coefficients} can usually be solved only approximately. In the following Lemma, we give a bound on the approximation error in terms of the difference in the objective function value.
\begin{lemma}[Inexact solution of the optimization problem]\label[lemma]{lem:inexact-solution-optimization-problem}
	Let us denote by~$u_h\in V_{X,k}\subset V$ the solution of~\eqref{equ:energy-minimization-kernel-approximation}. Let us further denote by~$\tilde{u}_h\in V_{X,k}$ an approximation of~$u_h$. Then it holds that
	\begin{align*}
		\norm{u_h-\tilde{u}_h}_V^2 \leq \frac{2}{\coercivity}\Big(\energy(\tilde{u}_h)-\energy(u_h)\Big),
	\end{align*}
	where~$\coercivity$ and~$\continuity$ denote the coercivity and continuity constants of~$a$, respectively.
\end{lemma}
\begin{proof}
	Due to the coercivity, symmetry and continuity of~$a$, as well as the assumption that~$u_h$ is a weak solution, it holds
	\begin{align*}
		\frac{\coercivity}{2}\norm{u_h-\tilde{u}_h}_V^2 &\leq \frac{1}{2}a(u_h-\tilde{u}_h,u_h-\tilde{u}_h) \\
		&= \frac{1}{2}a(u_h,u_h) - a(u_h,\tilde{u}_h) + \frac{1}{2}a(\tilde{u}_h,\tilde{u}_h) \\
		&= -\left[\frac{1}{2}a(u_h,u_h)-l(u_h)\right]+\frac{1}{2}a(\tilde{u}_h,\tilde{u}_h)-l(\tilde{u}_h)\\
		&\leq \energy(\tilde{u}_h) - \energy(u_h),
	\end{align*}
	where we used that~$u_h\in V_{X,k}$ is a weak solution and therefore~$a(u_h,\tilde{u}_h)=l(\tilde{u}_h)$ and
	\begin{align*}
		\frac{1}{2}a(u_h,u_h)=l(u_h)-\frac{1}{2}a(u_h,u_h).
	\end{align*}
	Dividing by~$\frac{\coercivity}{2}$ yields the claimed estimate.
\end{proof}
The previous Lemma provides a bound on how an inexact solution of the optimization problem in~\eqref{equ:optimization-problem-coefficients} can affect the approximation quality.
If the optimization problem is solved accurately, i.e.~$\energy(\tilde{u}_h)-\energy(u_h)$ is small, the approximation error of~$\tilde{u}_h$ compared to~$u_h$ is small as well.
However, it is important to note that the error of~$\tilde{u}_h$ with respect to the norm on~$V$ is bounded by the square root of the difference in the objective function values.
For a prescribed approximation error~$\varepsilon>0$, one thus has to solve the optimization problem up to a tolerance of around~$\varepsilon^2$.
Moreover, we remark that~$\energy(u_h)$ is usually not available in practice, such that the bound in~\Cref{lem:inexact-solution-optimization-problem} cannot be evaluated.
\par
We note that the result from~\Cref{lem:inexact-solution-optimization-problem} holds similarly for the weak solution~$u\in V$ of the variational problem~\eqref{equ:weak-formulation}, i.e.~we have
\begin{align*}
	\norm{u-\tilde{u}_h}_V^2 \leq \frac{2}{\coercivity}\Big(\energy(\tilde{u}_h)-\energy(u)\Big).
\end{align*}
However, we can also combine~\Cref{lem:inexact-solution-optimization-problem} with~\Cref{thm:a-priori-error-estimate-h1,thm:a-priori-error-estimate-l2} to obtain the following error bounds by means of the triangle inequality:
\begin{proposition}[A priori error estimates with inexact solution of the optimization problem]
	Let the assumptions from~\Cref{thm:a-priori-error-estimate-h1,thm:a-priori-error-estimate-l2} and~\Cref{lem:inexact-solution-optimization-problem} be fulfilled.
	Then we have that
	\begin{align*}
		\norm{u-\tilde{u}_h}_{H^1(\Omega)} \leq c_{H^1}h^{m-1}\norm{u}_{H^m(\Omega)}+\left[\frac{2}{\coercivity}\Big(\energy(\tilde{u}_h)-\energy(u_h)\Big)\right]^{1/2}
	\end{align*}
	and
	\begin{align*}
		\norm{u-\tilde{u}_h}_{L^2(\Omega)} \leq c_{L^2}h^m\norm{u}_{H^m(\Omega)}+\left[\frac{2}{\coercivity}\Big(\energy(\tilde{u}_h)-\energy(u_h)\Big)\right]^{1/2}.
	\end{align*}
\end{proposition}
In the next section we present two numerical examples in which we investigate the behavior of kernel methods in the deep Ritz framework and verify our theoretical convergence results in practice.

\subsection{The deep Ritz method for two-layered kernels}\label{sec:deep-ritz-two-layered}
In~\Cref{sec:kernel-methods}, an extension of traditional kernel approaches was briefly introduced, the so-called two-layered kernel methods.
Despite the coefficients of the kernel expansion, also a matrix within the nonlinear kernel is learnable in two-layered approaches.
Using two-layered kernels as ansatz functions for the weak problem can therefore not be written as a linear system in the coefficients due to the additional matrix entering the approximation in a nonlinear manner.
Similar to the motivation of the deep Ritz method for neural networks as described in~\Cref{sec:nns-and-deep-ritz}, we therefore suggest to use the energy minimization formulation to optimize coefficients and matrix of the kernel ansatz simultaneously.
Two-layered kernel approaches therefore provide an additional motivation to consider the energy minimization problem also for kernel methods.
\par
To make this precise, we extend the formulation of the approach in~\Cref{sec:description-of-method} by the trainable matrix~$B$ within the kernel.
With abuse of notation, let
\begin{align*}
	\psi_h[\boldsymbol{\coeffs},B](x) = \sum\limits_{i=1}^{n}\coeffs_i k_B(x,x_i)
\end{align*}
be a two-layered kernel expansion with fixed centers~$x_1,\ldots,x_n\in\R^d$, trainable coefficient vector~$\boldsymbol{\coeffs}\in\R^n$ and trainable matrix~$B\in\R^{d\times d}$.
We then solve the energy minimization problem
\begin{align*}
	(\boldsymbol{\coeffs}^*, B^*) = \argmin_{(\boldsymbol{\coeffs},B)\in\R^n\times\R^{d\times d}} \energy\big(\psi_h[\boldsymbol{\coeffs},B]\big)
\end{align*}
to obtain the approximation~$u_h=\psi_h[\boldsymbol{\coeffs}^*,B^*]$.
\par
We would like to emphasize that a priori error estimates as derived above do not immediately apply to the two-layered kernel case.
The kernel itself depends on the additional matrix and asymptotic results would require a fixed kernel (this is also the reason why we fix the kernel shape parameter in our experiment below).
However, training the matrix first for some set of centers and reusing this fixed matrix afterwards for different mesh norms would again allow to apply the theory developed above and we would expect similar convergence rates.
As long as the matrix is invertible, one can interpret the additional matrix within the kernel as a certain rescaling of the mesh norm which only influences the constants in the a priori convergence rates \cite{wenzel2024data}.
In order to keep the numerical results accessible, we refrain from conducting an experiment in this direction and focus on a simpler comparison of ``flat'' kernels, two-layered kernels and neural networks.
\par
We will see in~\Cref{sec:high-dimensional-example} that the two-layered structure is particularly useful for problems with anisotropic solution behavior in high dimensions.
For the low-dimensional problems in~\Cref{sec:high-regularity} and~\Cref{sec:singular-solution}, one cannot expect a superior performance of the two-layered methods and we therefore focus on evaluating the theoretical results derived for flat kernels in practice by means of these examples.

\section{Numerical experiments}\label{sec:numerical-experiments}
In this section we first discuss some implementational details and give additional information on the parameters used in our numerical examples. 
Afterwards, we present two different test cases in two spatial dimensions. 
First of all, we consider a problem where the solution is smooth such that the convergence rate of the deep Ritz approach using kernel methods is determined by the smoothness of the kernel. 
We will investigate this dependency of the convergence rate in detail and compare the results to our theoretical findings. 
Afterwards, we consider an example with a singular solution where we expect the convergence rates to be determined by the regularity of the solution independent of the smoothness of the kernel.
We finally present a high-dimensional example where we observe that two-layered kernels perform better than the traditional flat kernels.
As discussed in the previous section, we focus for the first two experiments on verifying the theoretical results using flat kernels and only consider two-layered approaches for the high-dimensional example.

\subsection{Experimental setup}\label{sec:implementational-details}
In order to (approximately) solve the optimization problem in~\eqref{equ:optimization-problem-coefficients}, 
it is possible to apply gradient-based methods such as the Adam optimizer~\cite{kingma2014adam} that is used in our numerical experiments. 
Various learning rate schedulers from the machine learning literature can further be used to adjust the step size within the gradient descent method. 
Below, we consider a learning rate scheduler that reduces the learning rate by a certain factor at prescribed milestones, 
i.e.~after given numbers of training epochs.
We perform up to~$10.000$ epochs and the learning rate scheduler reduces the learning rate by a factor of~$0.5$ when the loss does not decrease for~$200$ epochs.
We make use of the \texttt{PyTorch} library~\cite{paszke2019pytorch} and compute the gradient of the energy functional using automatic differentiation~\cite{baydin2017automatic}.
\par
In both of the numerical experiments below we make use of the weak formulation presented in~\Cref{ex:weak-formulation}.
The integrals in the energy functional~\eqref{equ:energy-functional} are approximated using Monte Carlo integration. 
The quadrature points are drawn randomly in every epoch, which corresponds to random mini-batching in the machine learning literature. 
We use in particular more quadrature points than centers and therefore ensure~$h_{\mathrm{int}}<h$, see~\Cref{sec:practical-aspects} as well. 
To be more precise, Monte Carlo quadrature has a convergence rate of~$\mathcal{O}(N^{-1/2})$ for~$N$ uniformly distributed points. 
Hence, the choice~$N\sim h^{-2(m+1)}$ leads to the same error rate as in~\Cref{thm:a-priori-error-estimate-h1} by Strang's~Lemma. 
However, this choice for the number of quadrature points is impractical already for moderate values of the mesh norm~$h$ and smoothness~$m$. 
We therefore choose a new quadrature point set of moderate size randomly in each epoch instead of generating a large set that is fixed during the whole optimization loop.
For the numerical experiments, we focus on the class of Mat\'{e}rn kernels~\cite{wendland2004scattered} of different smoothness,
which have an algebraic decay of the Fourier transform according to~\eqref{equ:Fourier-transform-decay}.
The regularity of Mat\'{e}rn kernels is determined by a parameter~$\nu>0$.
For~$\nu=p+\frac{1}{2}$ with~$p\in\N$, Mat\'{e}rn kernels possess a simplified expression as a product of an exponential and a polynomial of degree~$p$.
Hence, we consider in the following Mat\'{e}rn kernels for~$\nu\in\{1/2,3/2,5/2\}$, which correspond to an algebraic decay of the Fourier transform of order~$\tau=\nu+\frac{d}{2}$, see \eqref{equ:Fourier-transform-decay}.
Since we apply our approach to numerical examples in two dimensions, the order~$\tau$ is given as~$\tau=\nu+1\in\{3/2,5/2,7/2\}$.
The aforementioned relations regarding the regularity of the kernels together with the expected convergence rates according to~\Cref{thm:a-priori-error-estimate-h1,thm:a-priori-error-estimate-l2} are summarized in~\Cref{tab:matern-kernels-regularity}.
\par
\begin{table}[htbp]
	\centering
	\begin{tabular}{*{7}{c}}
		\toprule
		\multirow{2}{*}{$p$} & \multirow{2}{*}{$\nu$} & \multirow{2}{*}{$\tau$} & \multicolumn{2}{c}{\makecell{Convergence rate\\for regularity~$m\geq\tau$}} & \multicolumn{2}{c}{\makecell{Convergence rate\\for regularity~$m<\tau$}} \\
		\cmidrule(lr){4-5}\cmidrule(lr){6-7}
		& & & $L^2$-norm & $H^1$-norm & $L^2$-norm & $H^1$-norm \\ \midrule \midrule
		$0$ & $1/2$ & $3/2$ & $3/2$ & $1/2$ & $m$ & $m-1$ \\
		$1$ & $3/2$ & $5/2$ & $5/2$ & $3/2$ & $m$ & $m-1$ \\
		$2$ & $5/2$ & $7/2$ & $7/2$ & $5/2$ & $m$ & $m-1$ \\
		\bottomrule
	\end{tabular}
	\caption{Regularity of Mat\'{e}rn kernels for~$d=2$ with different parameters and convergence rates depending on the smoothness~$\tau$ of the kernel and the regularity~$m$ of the weak solution, i.e.~$k$ has an algebraically decaying Fourier transform of order~$\tau$ and we have~$u\in H^m(\Omega)$. For the convergence rates with respect to the~$L^2$-norm we disregard the assumption~$m\geq2$ from~\Cref{thm:a-priori-error-estimate-l2} in the overview in this table.}
	\label{tab:matern-kernels-regularity}
\end{table}
We also tested compactly supported kernels of finite smoothness such as the family of~Wendland kernels.
The results for these kernels can be found in~\Cref{app:numerical-results-kernels}.
\par
For improved convergence of the optimization, a change of basis of~$V_{X,k}$ to a Lagrange basis was done prior to the optimization.
We consider the minimization of the energy functional~\eqref{equ:energy-functional} where the penalty parameter was set to~$C_\text{pen}=100$ for both of the following examples, similar as in~\cite{e2018deepritz}.
The integrals that occur in the computation of errors are approximated by a uniform quadrature on the domain~$\Omega$ using~$10.201$ quadrature points.
We further depict relative errors of the form~$\norm{u - u_h} / \norm{u}$, where~$u$ denotes the exact solution and~$u_h$ the considered approximation.
The norm~$\norm{\,\cdot\,}$ refers either to the~$L^2$-norm or the~$H^1$-norm.
\par
In our numerical experiments in the next sections we also compare the approach of using kernel approximants to the original deep Ritz method with neural networks.
To this end, we trained fully-connected neural networks with two hidden layers and different numbers of neurons per layer.
The Gaussian Error Linear Unit~(GELU) was employed as activation function in the neural networks.
The choices of activation function and depth~$2$ are motivated by the hyperparameter experiments presented in the appendix of this paper.
Since typically no measure for spatial discretization in the context of neural networks that is comparable to the mesh norm for kernel methods is available, we use the number of parameters in order to compare both approaches.
In the case of kernel methods, the number of parameters refers to the expansion size, i.e.~the number of centers.
For neural networks, weights and biases are seen as parameters.
The neural networks are trained similarly to the kernel methods using the~Adam optimizer.
However, we train the neural networks using~$100.000$ epochs and an optimized learning rate scheduler that reduces the learning rate by a factor of~$0.5$ up to~$15$ times using a fixed schedule.
We chose the increased number of epochs for the neural networks because each optimization iteration was faster for the neural networks in our experiments.
The number of epochs was thus balanced in such a way that both approaches -- kernel methods and neural networks -- have roughly the same computational budget.
For the largest architectures and smallest mesh norm considered below, the training takes around~$25$~minutes on a~GPU.
\par
In the appendix we provide extensive hyperparameter sweeps for neural networks and kernel methods including different activation functions and architectures as well as different kernel families.
\par
The source code used to carry out the numerical experiments presented in this contribution can be found in~\cite{sourcecode} and can be used to reproduce the results shown below.
The results presented in this paper have been obtained on an NVIDIA~H200~GPU with~141GB memory and~4.8TB/s GPU memory bandwidth.

\subsection{Diffusion equation with a solution of high regularity}\label{sec:high-regularity}
As a first test case we consider the following Poisson problem with inhomogeneous Dirichlet boundary conditions as already introduced in~\Cref{ex:weak-formulation}:
\begin{equation}\label{equ:diffusion-problem-strong-form-deep-Ritz}
	\begin{aligned}
		-\Delta u &= f && \text{in }\Omega,\\
		u &= g && \text{on }\partial\Omega,
	\end{aligned}
\end{equation}
where~$\Omega=(0,1)^2$, $f(x)=4$ and~$g(x)=1-x_1^2-x_2^2$ for~$x=(x_1,x_2)\in\Omega$.
As can be easily verified, the strong solution of this problem is given as~$u=g$.
This problem thus possesses a solution with arbitrary high smoothness and we therefore expect that the convergence rate of the kernel approximation is limited by the regularity of the kernel, see~\Cref{thm:a-priori-error-estimate-h1}.
\par
Moreover, we consider~$n\in\{1,2,4,6,8,10,12,14,16,18,20\}$ centers per dimension in the interior of the domain~$\Omega$ and add~$4n+4$ points on the boundary.
The centers are placed on a uniform mesh which results in a mesh norm that scales like~$h\sim 1/(n+1)$ for~$n\in\N$ inner centers per dimension.
\par
\Cref{fig:errors-higher-regularity} depicts the errors of the deep Ritz method using Mat\'{e}rn kernels of different smoothness in the~$L^2$-norm and the~$H^1$-norm.
In addition, we also present the interpolation errors when considering the same set of centers as for the corresponding deep Ritz method and interpolating the true solution of~\eqref{equ:diffusion-problem-strong-form-deep-Ritz} in these centers.
For the interpolation error curves we further show an estimate of the convergence order using a linear regression of the data points in the asymptotic regime.
\par
\pgfplotstableread[skip first n=5]{results_interpolation_smooth_solution_errors_k_0.txt}\tab%
\pgfplotstablecreatecol[create col/expr={ln(\thisrowno{1})/ln(10)}]{logh}\tab%
\pgfplotstablecreatecol[create col/expr={ln(\thisrowno{3})/ln(10)}]{logL2}\tab%
\pgfplotstablecreatecol[linear regression={x=logh, y=logL2}]{regr}\tab%
\xdef\slopekzero{\pgfplotstableregressiona}\xdef\interceptkzero{\pgfplotstableregressionb}%
\pgfplotstableread[skip first n=5]{results_interpolation_smooth_solution_errors_k_1.txt}\tab%
\pgfplotstablecreatecol[create col/expr={ln(\thisrowno{1})/ln(10)}]{logh}\tab%
\pgfplotstablecreatecol[create col/expr={ln(\thisrowno{3})/ln(10)}]{logL2}\tab%
\pgfplotstablecreatecol[linear regression={x=logh, y=logL2}]{regr}\tab%
\xdef\slopekone{\pgfplotstableregressiona}\xdef\interceptkone{\pgfplotstableregressionb}%
\pgfplotstableread[skip first n=5]{results_interpolation_smooth_solution_errors_k_2.txt}\tab%
\pgfplotstablecreatecol[create col/expr={ln(\thisrowno{1})/ln(10)}]{logh}\tab%
\pgfplotstablecreatecol[create col/expr={ln(\thisrowno{3})/ln(10)}]{logL2}\tab%
\pgfplotstablecreatecol[linear regression={x=logh, y=logL2}]{regr}\tab%
\xdef\slopektwo{\pgfplotstableregressiona}\xdef\interceptktwo{\pgfplotstableregressionb}%
\pgfplotstableread[skip first n=5]{results_interpolation_smooth_solution_errors_k_0.txt}\tab%
\pgfplotstablecreatecol[create col/expr={ln(\thisrowno{1})/ln(10)}]{logh}\tab%
\pgfplotstablecreatecol[create col/expr={ln(\thisrowno{4})/ln(10)}]{logL2}\tab%
\pgfplotstablecreatecol[linear regression={x=logh, y=logL2}]{regr}\tab%
\xdef\slopekzeroh{\pgfplotstableregressiona}\xdef\interceptkzeroh{\pgfplotstableregressionb}%
\pgfplotstableread[skip first n=5]{results_interpolation_smooth_solution_errors_k_1.txt}\tab%
\pgfplotstablecreatecol[create col/expr={ln(\thisrowno{1})/ln(10)}]{logh}\tab%
\pgfplotstablecreatecol[create col/expr={ln(\thisrowno{4})/ln(10)}]{logL2}\tab%
\pgfplotstablecreatecol[linear regression={x=logh, y=logL2}]{regr}\tab%
\xdef\slopekoneh{\pgfplotstableregressiona}\xdef\interceptkoneh{\pgfplotstableregressionb}%
\pgfplotstableread[skip first n=5]{results_interpolation_smooth_solution_errors_k_2.txt}\tab%
\pgfplotstablecreatecol[create col/expr={ln(\thisrowno{1})/ln(10)}]{logh}\tab%
\pgfplotstablecreatecol[create col/expr={ln(\thisrowno{4})/ln(10)}]{logL2}\tab%
\pgfplotstablecreatecol[linear regression={x=logh, y=logL2}]{regr}\tab%
\xdef\slopektwoh{\pgfplotstableregressiona}\xdef\interceptktwoh{\pgfplotstableregressionb}%
\begin{figure}[htbp]%
	\begin{minipage}[b]{.48\textwidth}
		\centering%
		\begin{tikzpicture}
			\begin{axis}[
				DefaultStyle,
				log basis x={10},
				log basis y={10},
				xlabel={Mesh norm~$h$},
				xmin=0.04, xmax=0.55,
				xmode=log,
				xtick={0.5,0.25,0.1,0.05},
				xticklabels={0.5,0.25,0.1,0.05},
				ylabel={Relative $L^2$-error},
				ymin=1e-5,
				ymax=1,
				ymode=log,
				ytick={1e-08,1e-07,1e-06,1e-05,0.0001,0.001,0.01,0.1,1},
				legend style={at={(1.1,1.05)}, anchor=south},
				legend columns=3,
				transpose legend,
				width=\textwidth,
				]
				\addplot[domain=0.04:0.55, samples=2, thick, forget plot] {10^(\interceptkzero) * x^(\slopekzero)};
				\slopetrileft{\slopekzero}{\interceptkzero}{0.075}{0.2}{\pgfmathprintnumber{\slopekzero}}
				\addplot[thick, colork0, mark=x, mark size=3, mark options={solid}, dotted] table[x index=1, y index=3] {results_interpolation_smooth_solution_errors_k_0.txt};
				\addlegendentry{$\nu=1/2$: interpolation}
				\addplot[domain=0.04:0.55, samples=2, thick, forget plot] {10^(\interceptkone) * x^(\slopekone)};
				\slopetriright{\slopekone}{\interceptkone}{0.15}{0.275}{\pgfmathprintnumber{\slopekone}}
				\addplot[thick, colork1, mark=x, mark size=3, mark options={solid}, dotted] table[x index=1, y index=3] {results_interpolation_smooth_solution_errors_k_1.txt};
				\addlegendentry{$\nu=3/2$: interpolation}
				\addplot[domain=0.04:0.55, samples=2, thick, forget plot] {10^(\interceptktwo) * x^(\slopektwo)};
				\slopetriright{\slopektwo}{\interceptktwo}{0.075}{0.125}{\pgfmathprintnumber{\slopektwo}}
				\addplot[thick, colork2, mark=x, mark size=3, mark options={solid}, dotted] table[x index=1, y index=3] {results_interpolation_smooth_solution_errors_k_2.txt};
				\addlegendentry{$\nu=5/2$: interpolation}
				\addplot[thick, colork0, mark=*, mark size=2, mark options={solid}] table[x index=1, y index=3] {results_smooth_solution_errors_k_0.txt};
				\addlegendentry{$\nu=1/2$: deep Ritz with kernels}
				\addplot[thick, colork1, mark=*, mark size=2, mark options={solid}] table[x index=1, y index=3] {results_smooth_solution_errors_k_1.txt};
				\addlegendentry{$\nu=3/2$: deep Ritz with kernels}
				\addplot[thick, colork2, mark=*, mark size=2, mark options={solid}] table[x index=1, y index=3] {results_smooth_solution_errors_k_2.txt};
				\addlegendentry{$\nu=5/2$: deep Ritz with kernels}
			\end{axis}
		\end{tikzpicture}%
	\end{minipage}%
	\hfill%
	\begin{minipage}[b]{.48\textwidth}
		\centering%
		\begin{tikzpicture}
			\begin{axis}[
				DefaultStyle,
				log basis x={10},
				log basis y={10},
				xlabel={Mesh norm~$h$},
				xmin=0.04, xmax=0.55,
				xmode=log,
				xtick={0.5,0.25,0.1,0.05},
				xticklabels={0.5,0.25,0.1,0.05},
				ylabel={Relative $H^1$-error},
				ymin=1e-5,
				ymax=1,
				ymode=log,
				ytick={1e-08,1e-07,1e-06,1e-05,0.0001,0.001,0.01,0.1,1},
				width=\textwidth,
				]
				\addplot[domain=0.04:0.55, samples=2, thick, forget plot] {10^(\interceptkzeroh) * x^(\slopekzeroh)};
				\slopetrileft{\slopekzeroh}{\interceptkzeroh}{0.075}{0.2}{\pgfmathprintnumber{\slopekzeroh}}
				\addplot[thick, colork0, mark=x, mark size=3, mark options={solid}, dotted] table[x index=1, y index=4] {results_interpolation_smooth_solution_errors_k_0.txt};
				\addplot[thick, colork0, mark=*, mark size=2, mark options={solid}] table[x index=1, y index=4] {results_smooth_solution_errors_k_0.txt};
				\addplot[domain=0.04:0.55, samples=2, thick, forget plot] {10^(\interceptkoneh) * x^(\slopekoneh)};
				\slopetriright{\slopekoneh}{\interceptkoneh}{0.125}{0.275}{\pgfmathprintnumber{\slopekoneh}}
				\addplot[thick, colork1, mark=x, mark size=3, mark options={solid}, dotted] table[x index=1, y index=4] {results_interpolation_smooth_solution_errors_k_1.txt};
				\addplot[thick, colork1, mark=*, mark size=2, mark options={solid}] table[x index=1, y index=4] {results_smooth_solution_errors_k_1.txt};
				\addplot[domain=0.04:0.55, samples=2, thick, forget plot] {10^(\interceptktwoh) * x^(\slopektwoh)};
				\slopetriright{\slopektwoh}{\interceptktwoh}{0.05}{0.1}{\pgfmathprintnumber{\slopektwoh}}
				\addplot[thick, colork2, mark=x, mark size=3, mark options={solid}, dotted] table[x index=1, y index=4] {results_interpolation_smooth_solution_errors_k_2.txt};
				\addplot[thick, colork2, mark=*, mark size=2, mark options={solid}] table[x index=1, y index=4] {results_smooth_solution_errors_k_2.txt};
			\end{axis}
		\end{tikzpicture}%
	\end{minipage}%
	\caption{Relative errors in the~$L^2$-norm (left) and the~$H^1$-norm (right) with respect to the mesh norm~$h$ for~Mat\'{e}rn kernels of smoothness~$\nu \in \{1/2,3/2,5/2\}$ in the example with a smooth solution. The dashed lines indicate the error decays of the interpolation of the exact solution by the respective kernel. The estimated convergence rates for the interpolated solutions are shown as black lines.}
	\label{fig:errors-higher-regularity}
\end{figure}%
We observe that the interpolation of the exact solution and the approximation obtained using the deep Ritz approach show the same asymptotical converge rate. 
Only for~$k=2$ and a small mesh norm of less than~$0.01$, the error of the deep Ritz solution stagnates and does not decrease further when adding more centers.
We expect that this issue can be circumvented with more advanced learning rate scheduling.
We further notice that the~$L^2$-error of the deep Ritz solution is smaller than the interpolation error in most cases.
The deep Ritz solution is a quasi-best approximation with respect to the~$H^1$-norm whereas the interpolation only considers values of the exact solution at the centers.
Approximating the weak solution by means of the deep Ritz solution therefore takes much more information into account in its construction.
Moreover, the approximation error can indeed be smaller than the interpolation error, see also the discussion after~\Cref{thm:approximation-error-estimate}.
Certainly, the interpolation serves solely as a reference here.
In practice, the exact solution is usually unknown and it is impossible to compute its interpolation.
\par
In this example the solution is smooth and we have in particular that~$m>\tau$ for the considered values of~$\tau$.
For instance for~$\nu=1/2$, we therefore expect a convergence rate of~$3/2$ with respect to the~$L^2$-norm according to~\Cref{tab:matern-kernels-regularity}.
Indeed we observe in~\Cref{fig:errors-higher-regularity} a convergence rate of about~$2$, i.e.~an additional rate of~$1/2$.
This effect is an instance of superconvergence, see~\cite{schaback2018superconvergence,karvonen2025general}, 
which is a known phenomenon occurring in the context of kernel based approximation when the function that is approximated is much smoother than the kernel itself.
We thus obtain higher convergence rates as predicted by the theory in both norms and for all values of~$\nu$. 
\par
In~\Cref{fig:errors-higher-regularity-neural-networks} we present a comparison of the errors when using neural networks or Mat\'{e}rn kernels with regularity parameter~$\nu=3/2$ in the deep Ritz method.
\begin{figure}[htbp]%
	\begin{minipage}[b]{.48\textwidth}
		\centering%
		\begin{tikzpicture}
			\begin{axis}[
				DefaultStyle,
				log basis x={10},
				log basis y={10},
				xlabel={Number of parameters},
				xmin=1,
				xmax=500,
				xtick={100,200,300,400,500},
				ylabel={Relative $L^2$-error},
				ymin=1e-4,
				ymax=1,
				ymode=log,
				ytick={1e-08,1e-07,1e-06,1e-05,0.0001,0.001,0.01,0.1,1},
				legend style={at={(1.19,1.05)}, anchor=south},
				legend columns=2,
				width=\textwidth,
				]
				\addplot[thick, colork1, mark=*, mark size=2, mark options={solid}] table[x index=2, y index=3] {results_smooth_solution_errors_k_1.txt};
				\addlegendentry{$\nu=3/2$: deep Ritz with kernels}
				\addplot[thick, colorNN, mark=triangle*, mark size=2, mark options={solid}] table[x index=1, y index=2] {results_neural_network_smooth_solution_gelu_depth2_convergence_results.txt};
				\addlegendentry{deep Ritz with neural networks}
			\end{axis}%
		\end{tikzpicture}%
	\end{minipage}%
	\hfill%
	\begin{minipage}[b]{.48\textwidth}
		\centering%
		\begin{tikzpicture}
			\begin{axis}[
				DefaultStyle,
				log basis x={10},
				log basis y={10},
				xlabel={Number of parameters},
				xmin=1,
				xmax=500,
				xtick={100,200,300,400,500},
				ylabel={Relative $H^1$-error},
				ymin=1e-4,
				ymax=1,
				ymode=log,
				ytick={1e-08,1e-07,1e-06,1e-05,0.0001,0.001,0.01,0.1,1},
				legend style={at={(1.1,1.05)}, anchor=south},
				width=\textwidth,
				]
				\addplot[thick, colork1, mark=*, mark size=2, mark options={solid}] table[x index=2, y index=4] {results_smooth_solution_errors_k_1.txt};
				\addplot[thick, colorNN, mark=triangle*, mark size=2, mark options={solid}] table[x index=1, y index=3] {results_neural_network_smooth_solution_gelu_depth2_convergence_results.txt};
			\end{axis}%
		\end{tikzpicture}%
	\end{minipage}%
	\caption{Relative errors in the~$L^2$-norm (left) and the~$H^1$-norm (right) with respect to the number of parameters for~Mat\'{e}rn kernels of smoothness~$\nu=3/2$ and fully-connected neural networks in the example with a smooth solution.}
	\label{fig:errors-higher-regularity-neural-networks}
\end{figure}%
The relative errors of the neural networks in the~$L^2$-norm are about one order of magnitude larger in comparison to those obtained by the kernel models.
In the~$H^1$-norm, the errors are similar for both methods.
In particular for small numbers of parameters the neural networks seem to lack the expressivity in order to provide errors comparable to the kernel approach.
However, we also observe that in contrast to the error of the kernel methods, the errors of the neural networks still decrease further significantly when increasing the number of parameters in the network.
\par
We further consider the results obtained when assembling the stiffness matrix in~\eqref{equ:definition-stiffness-matrix} as well as the right-hand side and solve the system~\eqref{equ:system-stiffness-matrix} using Cholesky decomposition (the stiffness matrix is positive-definite and dense).
The linear system suffers from a very high condition number for smaller mesh norm, which is a well-known limitation of kernel methods that was discussed in more detail in~\Cref{sec:application-kernels-to-deep-ritz}.
We therefore add a Tikhonov regularization term with regularization factor~$10^{-10}$ when using Cholesky decomposition to solve the system.
Without this regularization, Cholesky decomposition failed in our experiments.
The relative errors with respect to the mesh norm when applying Mat\'{e}rn kernels with different smoothness are shown in~\Cref{fig:errors-higher-regularity-matrix}.
\par
\begin{figure}[htbp]%
	\begin{minipage}[b]{.48\textwidth}
		\centering%
		\begin{tikzpicture}
			\begin{axis}[
				DefaultStyle,
				log basis x={10},
				log basis y={10},
				xlabel={Mesh norm~$h$},
				xmin=0.04, xmax=0.55,
				xmode=log,
				ylabel={Relative $L^2$-error},
				ymin=1e-4,
				ymax=1,
				ymode=log,
				xtick={0.5,0.25,0.1,0.05},
				xticklabels={0.5,0.25,0.1,0.05},
				ytick={1e-08,1e-07,1e-06,1e-05,0.0001,0.001,0.01,0.1,1},
				legend style={at={(1.075,1.05)}, anchor=south},
				legend columns=3,
				transpose legend,
				width=\textwidth,
				]
				\addplot[thick, colork0, mark=x, mark size=3, mark options={solid}, dotted] table[x index=1, y index=3] {results_matrix_form_smooth_solution_errors_k_0.txt};
				\addlegendentry{$\nu=1/2$: solution of linear system}
				\addplot[thick, colork1, mark=x, mark size=3, mark options={solid}, dotted] table[x index=1, y index=3] {results_matrix_form_smooth_solution_errors_k_1.txt};
				\addlegendentry{$\nu=3/2$: solution of linear system}
				\addplot[thick, colork2, mark=x, mark size=3, mark options={solid}, dotted] table[x index=1, y index=3] {results_matrix_form_smooth_solution_errors_k_2.txt};
				\addlegendentry{$\nu=5/2$: solution of linear system}
				\addplot[thick, colork0, mark=*, mark size=2, mark options={solid}] table[x index=1, y index=3] {results_smooth_solution_errors_k_0.txt};
				\addlegendentry{$\nu=1/2$: deep Ritz with kernels}
				\addplot[thick, colork1, mark=*, mark size=2, mark options={solid}] table[x index=1, y index=3] {results_smooth_solution_errors_k_1.txt};
				\addlegendentry{$\nu=3/2$: deep Ritz with kernels}
				\addplot[thick, colork2, mark=*, mark size=2, mark options={solid}] table[x index=1, y index=3] {results_smooth_solution_errors_k_2.txt};
				\addlegendentry{$\nu=5/2$: deep Ritz with kernels}
			\end{axis}
		\end{tikzpicture}%
	\end{minipage}%
	\hfill%
	\begin{minipage}[b]{.48\textwidth}
		\centering%
		\begin{tikzpicture}
			\begin{axis}[
				DefaultStyle,
				log basis x={10},
				log basis y={10},
				xlabel={Mesh norm~$h$},
				xmin=0.04, xmax=0.55,
				xmode=log,
				xtick={0.5,0.25,0.1,0.05},
				xticklabels={0.5,0.25,0.1,0.05},
				ylabel={Relative $H^1$-error},
				ymin=1e-4,
				ymax=1,
				ymode=log,
				ytick={1e-08,1e-07,1e-06,1e-05,0.0001,0.001,0.01,0.1,1},
				width=\textwidth,
				]
				\addplot[thick, colork0, mark=x, mark size=3, mark options={solid}, dotted] table[x index=1, y index=4] {results_matrix_form_smooth_solution_errors_k_0.txt};
				\addplot[thick, colork0, mark=*, mark size=2, mark options={solid}] table[x index=1, y index=4] {results_smooth_solution_errors_k_0.txt};
				\addplot[thick, colork1, mark=x, mark size=3, mark options={solid}, dotted] table[x index=1, y index=4] {results_matrix_form_smooth_solution_errors_k_1.txt};
				\addplot[thick, colork1, mark=*, mark size=2, mark options={solid}] table[x index=1, y index=4] {results_smooth_solution_errors_k_1.txt};
				\addplot[thick, colork2, mark=x, mark size=3, mark options={solid}, dotted] table[x index=1, y index=4] {results_matrix_form_smooth_solution_errors_k_2.txt};
				\addplot[thick, colork2, mark=*, mark size=2, mark options={solid}] table[x index=1, y index=4] {results_smooth_solution_errors_k_2.txt};
			\end{axis}
		\end{tikzpicture}%
	\end{minipage}%
	\caption{Relative errors in the~$L^2$-norm (left) and the~$H^1$-norm (right) with respect to the mesh norm~$h$ for~Mat\'{e}rn kernels of smoothness~$\nu \in \{1/2,3/2,5/2\}$ using the energy minimization problem and the linear system of equations in the example with a smooth solution.}
	\label{fig:errors-higher-regularity-matrix}
\end{figure}%
We observe that already for mesh norms of about~$0.01$ the solution obtained by solving the linear system are worse than those computed by the deep Ritz method.
For smaller mesh sizes, the relative errors even start to increase which results in an error that is roughly three orders of magnitude larger than the corresponding error of the deep Ritz solution.
Moreover, this behavior is the same independent of the smoothness of the kernel.
For~$\nu=5/2$ the problem is even more pronounced than for~$\nu=1/2$ or~$\nu=3/2$.
The condition numbers of the stiffness matrix range from~$4\cdot10^3$ (for~$\nu=1/2$ and~$n=1$) to~$1.2\cdot10^{15}$ (for~$\nu=5/2$ and~$n=20$).
Solving a linear system with the ill-conditioned stiffness matrix thus constitutes an issue for the linear solver resulting in inaccurate approximations as apparent from~\Cref{fig:errors-higher-regularity-matrix}.
\par
The Adam optimizer as we use it here can be interpreted as an iterative linear solver for the underlying system.
We therefore compare it to the conjugate gradient~(CG) method~\cite{hestenes1952methods} that is one of the standard solvers for linear systems with symmetric, positive-definite system matrix.
The~CG~method iteratively minimizes the residual of the linear system measured in the~$A$-norm, where~$A$ is the (symmetric and positive-definite) system matrix.
In our case, the residual in the~$A$-norm deviates from the Dirichlet energy only by a constant factor, namely the energy of the exact solution of the linear system.
The~CG~method therefore also minimizes an equivalent energy as used for Adam in an iterative manner.
There is, however, one important difference between solving a fixed assembled system and the method we proposed here:
In every iteration of the optimization algorithm, we use a different random quadrature for computing the gradient of the Dirichlet energy, i.e.~we are essentially doing stochastic gradient descent.
For the matrix-based solution using Cholesky decomposition or the~CG~method, we first assemble the system matrix using a fixed quadrature rule and solve a fixed linear system with Tikhonov regularization factor~$10^{-10}$.
It is therefore of interest to investigate the role of the random quadrature and the iterative solver.
In~\Cref{fig:higher-regularity-loss-evolution}, we compare the loss (in terms of Dirichlet energy of the solution) over the iterations performed by the~CG~method and by the Adam optimizer.
For Adam, we performed the same experiment once using a fixed quadrature and once drawing random quadrature points in every epoch.
It turns out that Adam does not require any Tikhonov-type regularization in this case.
The experiment considers a fixed mesh norm of~$0.05$ and smoothness~$\nu=5/2$ for the~Mat\'{e}rn kernel.
We further summarize the final loss as well as~$L^2$- and~$H^1$-errors obtained by the different methods in~\Cref{tab:higher-regularity-iterative-methods-errors}.
\par
\begin{figure}[htbp]%
	\centering%
	\begin{tikzpicture}
		\begin{axis}[
			DefaultStyle,
			log basis x={10},
			log basis y={10},
			xlabel={Iterations/Epochs},
			xmin=1, xmax=10000,
			xmode=log,
			ylabel={Energy/Loss},
			ymin=-85,
			ymax=15,
			legend style={at={(0.5,1.05)}, anchor=south},
			legend columns=3,
			width=\textwidth,
			height=5cm,
			legend reversed=true,
			]
			\addplot[thick, viridisYellow] table[x index=0, y index=1, x expr=\thisrowno{0}+1] {results_smooth_solution_matern_k2_singlerun_loss_history_2_1.0_20_10000_reduced.txt};
			\addlegendentry{Adam (random quadrature)}
			\addplot[thick, viridisGreen] table[x index=0, y index=1, x expr=\thisrowno{0}+1] {results_smooth_solution_matern_k2_fixed_loss_history_2_1.0_20_10000_reduced.txt};
			\addlegendentry{Adam (fixed quadrature)}
			\addplot[thick, viridisBlue] table[x index=0, y index=1, x expr=\thisrowno{0}+1] {results_matrix_form_smooth_solution_cg_uniform_cg_energy_k_2_n_20_reduced.txt};
			\addlegendentry{CG (fixed quadrature)}
		\end{axis}
	\end{tikzpicture}%
	\caption{Dirichlet energy (loss) over the iterations for the conjugate gradient method, Adam with fixed quadrature and Adam with random quadrature in every epoch.}
	\label{fig:higher-regularity-loss-evolution}
\end{figure}
\begin{table}[htbp]%
	\centering%
	\begin{tabular}{c|ccc}
		\toprule
		Method & Loss & $L^2$-error & $H^1$-error \\
		\midrule\midrule
		CG (fixed quadrature) & $-73.37$ & $\hphantom{0}1.9\cdot10^{-2}$ & $\hphantom{0}1.3\cdot10^{-1}$ \\
		Adam (fixed quadrature) & $-73.52$ & $\hphantom{0}1.7\cdot10^{-1}$ & $2.06\cdot10^{0\hphantom{-}}$ \\
		Adam (random quadrature) & $-72.04$ & $2.26\cdot10^{-4}$ & $2.79\cdot10^{-3}$ \\
		\bottomrule
	\end{tabular}
	\caption{Dirichlet energy (loss) and relative errors in the~$L^2$-norm and the~$H^1$-norm at the final iteration for the conjugate gradient method, Adam with fixed quadrature and Adam with random quadrature in every epoch.}
	\label{tab:higher-regularity-iterative-methods-errors}
\end{table}
The optimization results over the iterations shows some interesting behaviors:
First of all, we note that using a fixed quadrature leads to both~CG and~Adam converging relatively fast after around~$100$ iterations to a similar loss value.
However, \Cref{tab:higher-regularity-iterative-methods-errors} reveals that the corresponding solutions are fundamentally different in terms of their errors.
The solution produced by Adam has errors around one order of magnitude larger compared to the solution of the~CG~method.
Looking at the coefficients in the respective kernel approximation, we observe that their magnitudes become very large when using Adam as optimizer.
These two observations together suggest that Adam overfits the training points while the solution oscillates outside of the training data.
When using the~CG~algorithm, this has been avoided by introducing a Tikhonov regularization term.
As a second important outcome, we remark that the loss of Adam when using random quadrature points in every iteration oscillates wildly over the optimization epochs.
This behavior is caused by the changing objective function in each iteration due to the random quadrature.
The loss drops even below~$-81$ in this case.
Nevertheless, we do not pick the kernel surrogate corresponding to the smallest loss since this might suffer from the same overfitting issue we observed for Adam with fixed quadrature.
We instead choose the approximation obtained in the last iteration of the optimization.
During the optimization with random quadrature, the optimizer makes use of information originating from many more quadrature points than it sees when fixing the quadrature a priori.
Although the loss at the final iteration is larger than for~CG or Adam with fixed quadrature, the~$L^2$- and~$H^1$-errors are significantly smaller when using stochastic quadrature rules, see~\Cref{tab:higher-regularity-iterative-methods-errors}.
Consequently, we conclude that the main benefit of our proposed method is that it naturally allows to explore many more quadrature points by sampling them in every iteration instead of using fixed ones and computing the stiffness matrix beforehand.
\par
In summary, the deep Ritz method using kernels provides superior results in terms of approximation errors compared to neural networks with a similar number of parameters for this example with a smooth solution.
Moreover, by using the formulation as an energy minimization problem we are able to circumvent the issue of a large condition number that arises when assembling the stiffness matrix.
The computational effort is further comparable for all learning-based methods discussed in this section.

\subsection{Singular solution on a non-convex domain}\label{sec:singular-solution}
In the previous example, the solution was smooth and thus the convergence rate of the deep Ritz method using kernels was limited by the regularity of the applied kernel.
Here, we now consider a problem where the solution is singular and the convergence rate is determined by the smoothness of the weak solution of the~PDE.
Let~$\Omega\subset\R^2$ be a non-convex domain defined as
\begin{align*}
	\Omega = \{(r\cos(\varphi),r\sin(\varphi))\in\R^2:0<r<1.5,0<\varphi<\alpha\}\subset\R^2.
\end{align*}
We further consider the same diffusion problem as in~\eqref{equ:diffusion-problem-strong-form-deep-Ritz} for~$f(x)=0$ and
\begin{align*}
	g(x) = \norm{x}_2^{1/\alpha}\cdot\sin\left(\frac{\arctan(x_2/x_1)}{\alpha}\right) + 1
\end{align*}
for~$x=(x_1,x_2)\in\bar{\Omega}$.
The solution is again given as~$u=g$.
We set the angle to~$\alpha=\frac{3\pi}{2}$.
Then, we have~$u\in H^{\frac{2}{3\pi}+1-\varepsilon}(\Omega)$ for all~$0<\varepsilon\leq\frac{2}{3\pi}+1$ and~$u\notin H^{l}(\Omega)$ for~$l\geq\frac{2}{3\pi}+1\approx 1.2122$.
A plot of the exact solution is shown in~\Cref{fig:solution-plot-singular}.
\par
\begin{figure}[htbp]%
	\centering
	\begin{tikzpicture}
		\begin{polaraxis}[
			width=.5\textwidth,
			tickwidth=0,
			xtick distance=45,
			separate axis lines,
			ytick={0.5,1,1.5},
			yticklabels={0.5,1,1.5},
			yticklabel shift=-15pt,
			colorbar,
			colormap name={viridis},
			axis on top=true,
			]
			\addplot3[surf, shader=interp, samples=75, y domain=0:1.5, domain=0:270] {y^(2/(3*pi))*sin(deg(2*(x*pi/180)/(3*pi))) + 1};
		\end{polaraxis}
	\end{tikzpicture}
	\caption{Top view plot of the exact solution in the example with a non-convex domain resulting in a singular solution. The singularity occurs in the origin due to steep gradients.}
	\label{fig:solution-plot-singular}
\end{figure}%
Similar to the previous example we consider~$n\in\{1,2,4,6,8,10,12,14,16,18,20\}$ centers per dimension in the square~$[-1.5,1.5]^2$ and remove those centers that are not in the interior of the domain~$\Omega$.
Further, we add~$4n+4$ points on the boundary~$\partial\Omega$.
\par
\Cref{fig:errors-singular} shows the decay of the relative~$L^2$-error when using~Mat\'{e}rn kernels of different smoothness to approximate the singular solution.
\pgfplotstableread[skip first n=5]{results_interpolation_singular_solution_errors_k_2.txt}\tab%
\pgfplotstablecreatecol[create col/expr={ln(\thisrowno{1})/ln(10)}]{logh}\tab%
\pgfplotstablecreatecol[create col/expr={ln(\thisrowno{3})/ln(10)}]{logL2}\tab%
\pgfplotstablecreatecol[linear regression={x=logh, y=logL2}]{regr}\tab%
\xdef\slopektwo{\pgfplotstableregressiona}\xdef\interceptktwo{\pgfplotstableregressionb}%
\pgfplotstableread[skip first n=5]{results_interpolation_singular_solution_errors_k_2.txt}\tab%
\pgfplotstablecreatecol[create col/expr={ln(\thisrowno{1})/ln(10)}]{logh}\tab%
\pgfplotstablecreatecol[create col/expr={ln(\thisrowno{4})/ln(10)}]{logL2}\tab%
\pgfplotstablecreatecol[linear regression={x=logh, y=logL2}]{regr}\tab%
\xdef\slopektwoh{\pgfplotstableregressiona}\xdef\interceptktwoh{\pgfplotstableregressionb}%
\begin{figure}[htbp]%
	\begin{minipage}[b]{.48\textwidth}
		\centering%
		\begin{tikzpicture}
			\begin{axis}[
				DefaultStyle,
				log basis x={10},
				log basis y={10},
				xlabel={Mesh norm~$h$},
				xmin=0.04, xmax=0.55,
				xmode=log,
				xtick={0.5,0.25,0.1,0.05},
				xticklabels={0.5,0.25,0.1,0.05},
				ylabel={Relative $L^2$-error},
				ymin=1e-3,
				ymax=0.25,
				ymode=log,
				ytick={1e-08,1e-07,1e-06,1e-05,0.0001,0.001,0.01,0.1},
				legend style={at={(1.1,1.05)}, anchor=south},
				legend columns=3,
				transpose legend,
				width=\textwidth,
				]
				\addplot[thick, colork0, mark=x, mark size=3, mark options={solid}, dotted] table[x index=1, y index=3] {results_interpolation_singular_solution_errors_k_0.txt};
				\addlegendentry{$\nu=1/2$: interpolation}
				\addplot[thick, colork1, mark=x, mark size=3, mark options={solid}, dotted] table[x index=1, y index=3] {results_interpolation_singular_solution_errors_k_1.txt};
				\addlegendentry{$\nu=3/2$: interpolation}
				\addplot[thick, colork2, mark=x, mark size=3, mark options={solid}, dotted] table[x index=1, y index=3] {results_interpolation_singular_solution_errors_k_2.txt};
				\addlegendentry{$\nu=5/2$: interpolation}
				\addplot[thick, colork0, mark=*, mark size=2, mark options={solid}] table[x index=1, y index=3] {results_singular_solution_errors_k_0.txt};
				\addlegendentry{$\nu=1/2$: deep Ritz with kernels}
				\addplot[thick, colork1, mark=*, mark size=2, mark options={solid}] table[x index=1, y index=3] {results_singular_solution_errors_k_1.txt};
				\addlegendentry{$\nu=3/2$: deep Ritz with kernels}
				\addplot[domain=0.04:0.55, samples=2, thick, forget plot] {10^(\interceptktwo) * x^(\slopektwo)};
				\slopetrileft{\slopektwo}{\interceptktwo}{0.075}{0.2}{\pgfmathprintnumber{\slopektwo}}
				\addplot[thick, colork2, mark=*, mark size=2, mark options={solid}] table[x index=1, y index=3] {results_singular_solution_errors_k_2.txt};
				\addlegendentry{$\nu=5/2$: deep Ritz with kernels}
			\end{axis}
		\end{tikzpicture}%
	\end{minipage}%
	\hfill%
	\begin{minipage}[b]{.48\textwidth}
		\centering%
		\begin{tikzpicture}
			\begin{axis}[
				DefaultStyle,
				log basis x={10},
				log basis y={10},
				xlabel={Mesh norm~$h$},
				xmin=0.04, xmax=0.55,
				xmode=log,
				xtick={0.5,0.25,0.1,0.05},
				xticklabels={0.5,0.25,0.1,0.05},
				ylabel={Relative $H^1$-error},
				ymin=5e-2,
				ymax=0.5,
				ymode=log,
				ytick={0.1,0.25},
				yticklabels={0.1,0.25},
				width=\textwidth,
				]
				\addplot[thick, colork0, mark=x, mark size=3, mark options={solid}, dotted] table[x index=1, y index=4] {results_interpolation_singular_solution_errors_k_0.txt};
				\addplot[thick, colork0, mark=*, mark size=2, mark options={solid}] table[x index=1, y index=4] {results_singular_solution_errors_k_0.txt};
				\addplot[thick, colork1, mark=x, mark size=3, mark options={solid}, dotted] table[x index=1, y index=4] {results_interpolation_singular_solution_errors_k_1.txt};
				\addplot[thick, colork1, mark=*, mark size=2, mark options={solid}] table[x index=1, y index=4] {results_singular_solution_errors_k_1.txt};
				\addplot[domain=0.04:0.55, samples=2, thick, forget plot] {10^(\interceptktwoh) * x^(\slopektwoh)};
				\slopetrileft{\slopektwoh}{\interceptktwoh}{0.075}{0.2}{\pgfmathprintnumber{\slopektwoh}}
				\addplot[thick, colork2, mark=x, mark size=3, mark options={solid}, dotted] table[x index=1, y index=4] {results_interpolation_singular_solution_errors_k_2.txt};
				\addplot[thick, colork2, mark=*, mark size=2, mark options={solid}] table[x index=1, y index=4] {results_singular_solution_errors_k_2.txt};
			\end{axis}
		\end{tikzpicture}%
	\end{minipage}%
	\caption{Relative errors in the~$L^2$-norm (left) and the~$H^1$-norm (right) with respect to the mesh norm~$h$ for~Mat\'{e}rn kernels of smoothness~$\nu \in \{1/2,3/2,5/2\}$ in the example with a singular solution. The dashed lines indicate the error decays of the interpolation of the exact solution by the respective kernel. The estimated convergence rate for the interpolation with~$\nu=5/2$ is shown as a black line.}
	\label{fig:errors-singular}
\end{figure}%
As we are in the ``escaping the native space'' case (smoothness of the solution smaller than smoothness of the kernel), 
we expect the error to decay according to the smoothness of the solution and independent of the smoothness of the kernel function.
This behavior can be observed in~\Cref{fig:errors-singular}, as the three plots for the different degrees of smoothness follow the same asymptotic.
We observe a decay rate of the error with respect to the~$L^2$-norm of~$1.4$ which roughly agrees with the regularity of the solution~$u$ of about~$1.21$.
Similarly as in the previous example, the decay rate with respect to the~$H^1$-norm is one order smaller than for the~$L^2$-norm.
The superconvergence that was present in the previous example cannot be observed here since the exact solution has a lower regularity than the kernels.
However, although the solution is singular, it can still be approximated accurately by smooth kernels.
\par
In the beginning (for mesh norms between~$2 \cdot 10^{-2}$ and~$5 \cdot 10^{-2}$) one can observe a preasymptotic behavior,
which is due to the small amount of centers that is not sufficient to properly cover the domain.
The results in~\Cref{fig:errors-singular} also show again that the solution obtained by the deep Ritz method results in smaller errors compared to the interpolation of the exact solution.
\par
In~\Cref{fig:errors-singular-neural-networks} we depict the relative errors for neural networks and Mat\'{e}rn kernels with~$\nu=3/2$ when used in the deep Ritz method.
\begin{figure}[htbp]%
	\begin{minipage}[b]{.48\textwidth}
		\centering%
		\begin{tikzpicture}
			\begin{axis}[
				DefaultStyle,
				log basis x={10},
				log basis y={10},
				xlabel={Number of parameters},
				xmin=1,
				xmax=500,
				xtick={100,200,300,400,500},
				ylabel={Relative $L^2$-error},
				ymin=1e-3,
				ymax=0.25,
				ymode=log,
				ytick={1e-08,1e-07,1e-06,1e-05,0.0001,0.001,0.01,0.1},
				legend style={at={(1.19,1.05)}, anchor=south},
				legend columns=2,
				width=\textwidth,
				]
				\addplot[thick, colork1, mark=*, mark size=2, mark options={solid}] table[x index=2, y index=3] {results_singular_solution_errors_k_1.txt};
				\addlegendentry{$\nu=3/2$: deep Ritz with kernels}
				\addplot[thick, colorNN, mark=triangle*, mark size=2, mark options={solid}] table[x index=1, y index=2] {results_neural_network_singular_solution_gelu_depth2_convergence_results.txt};
				\addlegendentry{deep Ritz with neural networks}
			\end{axis}%
		\end{tikzpicture}%
	\end{minipage}%
	\hfill%
	\begin{minipage}[b]{.48\textwidth}
		\centering%
		\begin{tikzpicture}
			\begin{axis}[
				DefaultStyle,
				log basis x={10},
				log basis y={10},
				xlabel={Number of parameters},
				xmin=1,
				xmax=500,
				xtick={100,200,300,400,500},
				ylabel={Relative $H^1$-error},
				ymin=5e-2,
				ymax=0.5,
				ymode=log,
				ytick={0.1,0.25},
				yticklabels={0.1,0.25},
				legend style={at={(1.1,1.05)}, anchor=south},
				width=\textwidth,
				]
				\addplot[thick, colork1, mark=*, mark size=2, mark options={solid}] table[x index=2, y index=4] {results_singular_solution_errors_k_1.txt};
				\addplot[thick, colorNN, mark=triangle*, mark size=2, mark options={solid}] table[x index=1, y index=3] {results_neural_network_singular_solution_gelu_depth2_convergence_results.txt};
			\end{axis}%
		\end{tikzpicture}%
	\end{minipage}%
	\caption{Relative errors in the~$L^2$-norm (left) and the~$H^1$-norm (right) with respect to the number of parameters for~Mat\'{e}rn kernels of smoothness~$\nu=3/2$ and fully-connected neural networks in the example with a singular solution.}
	\label{fig:errors-singular-neural-networks}
\end{figure}%
In this example, the kernel approximants and the neural networks perform quite similar. The~$H^1$-error of the neural networks is slightly smaller compared to the kernel methods.
We attribute this to the adaptive nature of the neural networks,
which can better deal with the localized singularity.
\par
Finally, we again investigate the performance of the solution computed by solving the weak formulation via the linear system in~\eqref{equ:system-stiffness-matrix}. In~\Cref{fig:errors-singular-matrix} we compare the results to the energy minimization approach using kernels.
\begin{figure}[htbp]%
	\begin{minipage}[b]{.48\textwidth}
		\centering%
		\begin{tikzpicture}
			\begin{axis}[
				DefaultStyle,
				log basis x={10},
				log basis y={10},
				xlabel={Mesh norm~$h$},
				xmin=0.04, xmax=0.55,
				xmode=log,
				xtick={0.5,0.25,0.1,0.05},
				xticklabels={0.5,0.25,0.1,0.05},
				ylabel={Relative $L^2$-error},
				ymin=1e-3,
				ymax=0.25,
				ymode=log,
				ytick={1e-08,1e-07,1e-06,1e-05,0.0001,0.001,0.01,0.1},
				legend style={at={(1.075,1.05)}, anchor=south},
				legend columns=3,
				transpose legend,
				width=\textwidth,
				]
				\addplot[thick, colork0, mark=x, mark size=3, mark options={solid}, dotted] table[x index=1, y index=3] {results_matrix_form_singular_solution_errors_k_0.txt};
				\addlegendentry{$\nu=1/2$: solution of linear system}
				\addplot[thick, colork1, mark=x, mark size=3, mark options={solid}, dotted] table[x index=1, y index=3] {results_matrix_form_singular_solution_errors_k_1.txt};
				\addlegendentry{$\nu=3/2$: solution of linear system}
				\addplot[thick, colork2, mark=x, mark size=3, mark options={solid}, dotted] table[x index=1, y index=3] {results_matrix_form_singular_solution_errors_k_2.txt};
				\addlegendentry{$\nu=5/2$: solution of linear system}
				\addplot[thick, colork0, mark=*, mark size=2, mark options={solid}] table[x index=1, y index=3] {results_singular_solution_errors_k_0.txt};
				\addlegendentry{$\nu=1/2$: deep Ritz with kernels}
				\addplot[thick, colork1, mark=*, mark size=2, mark options={solid}] table[x index=1, y index=3] {results_singular_solution_errors_k_1.txt};
				\addlegendentry{$\nu=3/2$: deep Ritz with kernels}
				\addplot[thick, colork2, mark=*, mark size=2, mark options={solid}] table[x index=1, y index=3] {results_singular_solution_errors_k_2.txt};
				\addlegendentry{$\nu=5/2$: deep Ritz with kernels}
			\end{axis}
		\end{tikzpicture}%
	\end{minipage}%
	\hfill%
	\begin{minipage}[b]{.48\textwidth}
		\centering%
		\begin{tikzpicture}
			\begin{axis}[
				DefaultStyle,
				log basis x={10},
				log basis y={10},
				xlabel={Mesh norm~$h$},
				xmin=0.04, xmax=0.55,
				xmode=log,
				xtick={0.5,0.25,0.1,0.05},
				xticklabels={0.5,0.25,0.1,0.05},
				ylabel={Relative $H^1$-error},
				ymin=5e-2,
				ymax=0.5,
				ymode=log,
				ytick={0.1,0.25},
				yticklabels={0.1,0.25},
				width=\textwidth,
				]
				\addplot[thick, colork0, mark=x, mark size=3, mark options={solid}, dotted] table[x index=1, y index=4] {results_matrix_form_singular_solution_errors_k_0.txt};
				\addplot[thick, colork0, mark=*, mark size=2, mark options={solid}] table[x index=1, y index=4] {results_singular_solution_errors_k_0.txt};
				\addplot[thick, colork1, mark=x, mark size=3, mark options={solid}, dotted] table[x index=1, y index=4] {results_matrix_form_singular_solution_errors_k_1.txt};
				\addplot[thick, colork1, mark=*, mark size=2, mark options={solid}] table[x index=1, y index=4] {results_singular_solution_errors_k_1.txt};
				\addplot[thick, colork2, mark=x, mark size=3, mark options={solid}, dotted] table[x index=1, y index=4] {results_matrix_form_singular_solution_errors_k_2.txt};
				\addplot[thick, colork2, mark=*, mark size=2, mark options={solid}] table[x index=1, y index=4] {results_singular_solution_errors_k_2.txt};
			\end{axis}
		\end{tikzpicture}%
	\end{minipage}%
	\caption{Relative errors in the~$L^2$-norm (left) and the~$H^1$-norm (right) with respect to the mesh norm~$h$ for~Mat\'{e}rn kernels of smoothness~$\nu \in \{1/2,3/2,5/2\}$ using the energy minimization problem and the linear system of equations in the example with a singular solution.}
	\label{fig:errors-singular-matrix}
\end{figure}%
The errors of both approaches, based on the linear system and the energy optimization, decay at the same rate.
Compared to the example in~\Cref{sec:high-regularity} we further observe that only for the smallest mesh norm and~$\nu=5/2$ the error of the solution based on the linear system is a bit larger.
The condition numbers in this example vary between~$2.7\cdot10^2$ and~$4.4\cdot10^{14}$.
However, in this case the effect seems to be less pronounced compared to the example with a smooth solution.
\par
This example shows that the deep Ritz method using kernels as ansatz functions can also be applied in settings where the weak solution has only low regularity.
Furthermore, the convergence rate is not influenced by the smoothness of the kernel in this case, which allows to apply kernels with high smoothness as well.
We observed that the kernel methods performed similar as the neural networks in this case of a singular solution.
However, we emphasize again that the observed convergence rate is underpinned by our theoretical results from~\Cref{sec:a-priori-error-estimation}.

\subsection{High-dimensional example}\label{sec:high-dimensional-example}
In this subsection we consider the high-dimensional example of \cite[Section 4.3]{haasdonk2025kernel}:
The problem of interest is the Poisson problem 
\begin{align*}
	-\Delta u(x) = \|\kappa\|_2^2\, \sin\left(x^\top\kappa\right), \qquad\text{for } x\in \Omega,
\end{align*}
on the domain~$\Omega = [0, 1]^{10}$. 
The solution is given by
\begin{align*}
	u(x) = \sin\left(x^\top\kappa\right), \qquad \text{for }x\in \Omega.
\end{align*}
The vector~$\kappa \in \R^{10}$ induces a strong directional dependence and is chosen as~$\frac{\pi}{10} (1,\ldots,1)^\top \in \R^{10}$.
\par
We apply the proposed method to this high-dimensional problem,
comparing standard kernel models, two-layered kernel models and neural networks.
Due to the multilayer setup, the two-layered kernel models and the neural networks are able to automatically
adapt to the strong directional dependence during optimization.
Therefore these two methods are expected to outperform standard radial kernels.
\par
In ten dimensions, using a tensor-product grid of uniform centers would result in the curse of dimensionality.
For both the standard and the two-layered kernel models, we therefore sample points in~$\Omega$ using the~$P$-greedy algorithm, which is known to produce well distributed points~\cite{wenzel2021novel}.
\par
For the kernel models, we tested Mat\'{e}rn kernels with shape parameter~$0.1$ using three different smoothness parameters~$\nu\in\{1/2,3/2,5/2\}$.
For the neural networks, we used as before a depth~$2$ network with~GELU activation function.
We tried neural networks of widths~$8$, $16$ and~$32$ per layer and kernel expansions with~$10$, $20$ or~$30$ centers, but only report the best model among all expansion sizes (width respectively number of centers) in order to keep the results concise.
The influence of these parameters turned out to be rather insignificant in this example.
\par
An overview of the results is presented in~\Cref{tab:errors-high-dimensional}:
One can observe that the standard kernel models are not able to reach a good accuracy, 
which can be attributed to their radial nature.
Both the two-layered kernels and the neural networks achieve significantly better results,
with best~$L^2$-errors of about~$4 \cdot 10^{-3}$ and~$5 \cdot 10^{-3}$ respectively.
\begin{table}[htbp]%
	\centering
	\begin{tabular}{l|ccc}
		\toprule
		Method & $\nu$ & $L^2$-error & $H^1$-error \\
		\midrule\midrule
		Flat kernel & $1/2$ & $5.568 \cdot 10^{-2}$ & $2.678 \cdot 10^{-1}$ \\
		Flat kernel & $3/2$ & $5.650 \cdot 10^{-2}$ & $2.744 \cdot 10^{-1}$ \\
		Flat kernel & $5/2$ & $5.670 \cdot 10^{-2}$ & $2.760 \cdot 10^{-1}$ \\
		\midrule
		Two-layered kernel & $1/2$ & $7.136 \cdot 10^{-3}$ & $4.777 \cdot 10^{-2}$ \\
		Two-layered kernel & $3/2$ & $4.458 \cdot 10^{-3}$ & $3.120 \cdot 10^{-2}$ \\
		Two-layered kernel & $5/2$ & $3.868 \cdot 10^{-3}$ & $2.706 \cdot 10^{-2}$ \\
		\midrule
		Neural networks & -- & $4.659 \cdot 10^{-3}$ & $3.122 \cdot 10^{-2}$ \\
		\bottomrule
	\end{tabular}
	\caption{Best result per expansion size (selected by lowest loss across all expansion sizes) for the high-dimensional example.}
	\label{tab:errors-high-dimensional}
\end{table}

\section{Conclusion and outlook}\label{sec:outlook-conclusion}
In this contribution we proposed the application of kernel methods as ansatz functions in the deep Ritz approach and analyzed its behavior theoretically as well as numerically.
In the original work~\cite{e2018deepritz} where the deep Ritz approach was introduced, the authors used deep neural networks and optimized the weights of the networks via minimization of an energy functional that is equivalent to the weak formulation of an elliptic~PDE.
In contrast, we employ kernel models and determine their coefficients by solving the optimization problem.
We moreover propose to use two-layered kernel models in particular for high-dimensional and structured problems, which introduce a nonlinear dependency on their parameters and therefore require optimization techniques instead of solvers for linear systems.
\par
The equivalence to the weak formulation allows for rigorous error analysis using standard techniques from numerical analysis.
We thus obtain a priori convergence estimates for the proposed method.
Such a priori error estimates are hard to derive for neural networks since the set of neural networks does not form a linear space and therefore results such as~C\'{e}a's~Lemma are not applicable in that case.
For kernel methods we prove convergence results with respect to the~$H^1$-norm and the~$L^2$-norm in this contribution.
These results show that the convergence rate of the approximation in terms of the mesh norm depends on the smoothness of the weak solution and on the regularity of the kernel.
It is therefore possible to obtain smaller errors by reducing the mesh norm, whereas for neural networks it is typically unclear what part of the neural network architecture to adjust, i.e.~whether to increase the number of layers or the number of neurons or whether to choose a different activation function.
Additionally, the solution can even be improved locally in space by adding more centers in certain parts of the domain -- and we aim to follow up with this, e.g.~by leveraging greedy kernel methods~\cite{wenzel2023analysis}.
Certainly, kernel methods involve several hyperparameters as well, such as the shape parameter, but the convergence results are independent of these parameters and are asymptotically only influenced by the smoothness of the kernel.
\par
We examined the performance of kernel methods within the deep Ritz approach by means of three numerical examples.
In the first experiment we considered a~PDE with a smooth solution such that the convergence rate is limited by the smoothness of the kernel according to~\Cref{thm:a-priori-error-estimate-h1,thm:a-priori-error-estimate-l2}.
The second experiment is based on a non-convex domain with a re-entrant corner that leads to a singular solution with low regularity.
As expected, the numerical experiment shows that the convergence rate is independent of the smoothness of the kernel and instead determined by the regularity of the solution.
In the second case, we moreover observe that neural networks perform similarly or even slightly better than kernel methods, whereas in the first case kernel methods are preferable to neural networks due to the smoothness of the solution.
In the third experiment, we consider a ten-dimensional spatial domain, which renders the usage of traditional methods such as finite elements infeasible.
This test case shows that two-layered kernel approaches are capable of achieving similar errors than neural networks also in high-dimensional examples, whereas flat kernel approaches might have limited applicability in those scenarios.
In all of our numerical experiments, the issue of large condition numbers of the stiffness matrix arises, which motivates again to solve the energy minimization problem instead of the linear system in the case of kernels as ansatz functions.
As we have seen numerically, randomly sampling the quadrature points in every optimization step improves the performance significantly compared to using a fixed system matrix.
\par
Meshless methods such as the kernel approaches based on the deep Ritz ansatz as described in this paper circumvent the necessity of computing a computational mesh.
Furthermore, the algorithms described in this contribution can be implemented with relatively small effort.
A (local) refinement of the solution is also readily possible by adding more centers.
Finite element methods in contrast would require a mesh refinement or a higher order approach to improve the solution.
\par
As potential extensions of the approach described here, one might for instance consider other problem classes, as already mentioned in~\Cref{rem:non-symmetric-problems}.
Moreover, the deep Ritz method has recently been considered for a nonlinear elliptic~PDE involving the~$p$-Laplacian~\cite{kaltenbach2025deepRitz}.
Such~PDEs possess an energy minimization formulation but do not allow for an equivalent linear system.
In such cases, using kernel methods with an iterative optimizer of the energy might be of interest.
\par
Another application where in particular two-layered kernels could become useful are parametric problems.
The parameter might be treated as an additional input component of the kernel approximant.
As objective function one could consider the integral of the parameter-dependent energy functionals over the parameter.
We assume that an anisotropic kernel would be indispensable for these cases in order to account for different scales in space and parameter.

\section*{Acknowledgements}

The authors would like to thank Lukas Renelt for providing ideas, in particular regarding~\Cref{lem:inexact-solution-optimization-problem}, and Julia Schleu\ss{} for several inspiring discussions.
Furthermore, the authors thank the anonymous reviewers for providing excellent comments that led to new insights and improved the manuscript significantly.

\section*{Declarations}

\paragraph{Funding}
\begin{itemize}
	\item H.~Kleikamp acknowledges funding by the Deutsche Forschungsgemeinschaft (DFG, German Research Foundation) under Germany's Excellence Strategy EXC 2044 --390685587, Mathematics M\"{u}nster: Dynamics--Geometry--Structure.
	\item T.~Wenzel acknowledges financial support through the projects LD-SODA of the {\em Landes\-for\-schungs\-f\"or\-de\-rung Hamburg} (LFF) and support from the RTG~2583 ``Modeling, Simulation and Optimization of Fluid Dynamic Applications'' funded by the {\em Deutsche Forschungsgemeinschaft} (DFG) and funding by Daimler and Benz Foundation as part of the scholarship program for junior professors and postdoctoral researchers.
\end{itemize}

\paragraph{Conflict of interest/Competing interests}
The authors have no relevant financial or non-financial interests to disclose.

\paragraph{Data availability}
The data used to produce the plots for the numerical experiments is provided in the source code repository, see~\cite{sourcecode}.

\paragraph{Code availability} The source code used to carry out the numerical experiments presented in this contribution can be found in~\cite{sourcecode}.

\paragraph{Author contribution}

\begin{itemize}
	\item H.~Kleikamp: Conceptualization, Methodology, Software, Validation, Formal analysis, Investigation, Writing -- Original Draft, Writing -- Review \& Editing, Visualization
	\item T.~Wenzel: Conceptualization, Methodology, Software, Validation, Formal analysis, Investigation, Writing -- Original Draft, Writing -- Review \& Editing, Visualization
\end{itemize}

\appendix
\section{Numerical results for Wendland kernels}\label{app:numerical-results-kernels}
In this part of the appendix, we show the numerical results for the family of Wendland kernels~\cite{wendland2004scattered}.
The regularity of a Wendland kernel is determined by a parameter~$p\in\N$ and the resulting native space is equivalent to the Sobolev space~$H^{d/2+k+1/2}(\R^d)$.
We consider~$p\in\{0,1,2\}$ and obtain the following expected convergence rates, summarized in~\Cref{tab:wendland-kernels-regularity}, which are similar to the ones for the~Mat\'{e}rn kernels, see~\Cref{tab:matern-kernels-regularity}.
\par
\begin{table}[htbp]
	\centering
	\begin{tabular}{*{6}{c}}
		\toprule
		\multirow{2}{*}{$p$} & \multirow{2}{*}{$\tau$} & \multicolumn{2}{c}{\makecell{Convergence rate\\for regularity~$m\geq\tau$}} & \multicolumn{2}{c}{\makecell{Convergence rate\\for regularity~$m<\tau$}} \\
		\cmidrule(lr){3-4}\cmidrule(lr){5-6}
		& & $L^2$-norm & $H^1$-norm & $L^2$-norm & $H^1$-norm \\ \midrule \midrule
		$0$ & $3/2$ & $3/2$ & $1/2$ & $m$ & $m-1$ \\
		$1$ & $5/2$ & $5/2$ & $3/2$ & $m$ & $m-1$ \\
		$2$ & $7/2$ & $7/2$ & $5/2$ & $m$ & $m-1$ \\
		\bottomrule
	\end{tabular}
	\caption{Regularity of Wendland kernels for~$d=2$ with different parameters and convergence rates depending on the smoothness~$\tau$ of the kernel and the regularity~$m$ of the weak solution, i.e.~$k$ has an algebraically decaying Fourier transform of order~$\tau$ and we have~$u\in H^m(\Omega)$. For the convergence rates with respect to the~$L^2$-norm we disregard the assumption~$m\geq2$ from~\Cref{thm:a-priori-error-estimate-l2} in the overview in this table.}
	\label{tab:wendland-kernels-regularity}
\end{table}
Apart from the choice of the kernel, all training parameters are chosen the same as in~\Cref{sec:numerical-experiments}.
In the following subsections, we briefly summarize the results obtained for the examples from above when applying Wendland kernels of different regularities.

\subsection{Diffusion equation with a solution of high regularity}
We first consider the example from~\Cref{sec:high-regularity} where the solution is smooth and the convergence rates of the deep Ritz method is dominated by the kernel regularity.
The relative errors in~$L^2$- and~$H^1$-norm are shown in~\Cref{fig:errors-higher-regularity-wendland}.
\pgfplotstableread[skip first n=5]{results_interpolation_smooth_solution_wendland_errors_k_0.txt}\tab%
\pgfplotstablecreatecol[create col/expr={ln(\thisrowno{1})/ln(10)}]{logh}\tab%
\pgfplotstablecreatecol[create col/expr={ln(\thisrowno{3})/ln(10)}]{logL2}\tab%
\pgfplotstablecreatecol[linear regression={x=logh, y=logL2}]{regr}\tab%
\xdef\slopekzero{\pgfplotstableregressiona}\xdef\interceptkzero{\pgfplotstableregressionb}%
\pgfplotstableread[skip first n=5]{results_interpolation_smooth_solution_wendland_errors_k_1.txt}\tab%
\pgfplotstablecreatecol[create col/expr={ln(\thisrowno{1})/ln(10)}]{logh}\tab%
\pgfplotstablecreatecol[create col/expr={ln(\thisrowno{3})/ln(10)}]{logL2}\tab%
\pgfplotstablecreatecol[linear regression={x=logh, y=logL2}]{regr}\tab%
\xdef\slopekone{\pgfplotstableregressiona}\xdef\interceptkone{\pgfplotstableregressionb}%
\pgfplotstableread[skip first n=5]{results_interpolation_smooth_solution_wendland_errors_k_2.txt}\tab%
\pgfplotstablecreatecol[create col/expr={ln(\thisrowno{1})/ln(10)}]{logh}\tab%
\pgfplotstablecreatecol[create col/expr={ln(\thisrowno{3})/ln(10)}]{logL2}\tab%
\pgfplotstablecreatecol[linear regression={x=logh, y=logL2}]{regr}\tab%
\xdef\slopektwo{\pgfplotstableregressiona}\xdef\interceptktwo{\pgfplotstableregressionb}%
\pgfplotstableread[skip first n=5]{results_interpolation_smooth_solution_wendland_errors_k_0.txt}\tab%
\pgfplotstablecreatecol[create col/expr={ln(\thisrowno{1})/ln(10)}]{logh}\tab%
\pgfplotstablecreatecol[create col/expr={ln(\thisrowno{4})/ln(10)}]{logL2}\tab%
\pgfplotstablecreatecol[linear regression={x=logh, y=logL2}]{regr}\tab%
\xdef\slopekzeroh{\pgfplotstableregressiona}\xdef\interceptkzeroh{\pgfplotstableregressionb}%
\pgfplotstableread[skip first n=5]{results_interpolation_smooth_solution_wendland_errors_k_1.txt}\tab%
\pgfplotstablecreatecol[create col/expr={ln(\thisrowno{1})/ln(10)}]{logh}\tab%
\pgfplotstablecreatecol[create col/expr={ln(\thisrowno{4})/ln(10)}]{logL2}\tab%
\pgfplotstablecreatecol[linear regression={x=logh, y=logL2}]{regr}\tab%
\xdef\slopekoneh{\pgfplotstableregressiona}\xdef\interceptkoneh{\pgfplotstableregressionb}%
\pgfplotstableread[skip first n=5]{results_interpolation_smooth_solution_wendland_errors_k_2.txt}\tab%
\pgfplotstablecreatecol[create col/expr={ln(\thisrowno{1})/ln(10)}]{logh}\tab%
\pgfplotstablecreatecol[create col/expr={ln(\thisrowno{4})/ln(10)}]{logL2}\tab%
\pgfplotstablecreatecol[linear regression={x=logh, y=logL2}]{regr}\tab%
\xdef\slopektwoh{\pgfplotstableregressiona}\xdef\interceptktwoh{\pgfplotstableregressionb}%
\begin{figure}[htbp]%
	\begin{minipage}[b]{.48\textwidth}
		\centering%
		\begin{tikzpicture}
			\begin{axis}[
				DefaultStyle,
				log basis x={10},
				log basis y={10},
				xlabel={Mesh norm~$h$},
				xmin=0.04, xmax=0.55,
				xmode=log,
				xtick={0.5,0.25,0.1,0.05},
				xticklabels={0.5,0.25,0.1,0.05},
				ylabel={Relative $L^2$-error},
				ymin=1e-4,
				ymax=1,
				ymode=log,
				ytick={1e-08,1e-07,1e-06,1e-05,0.0001,0.001,0.01,0.1,1},
				legend style={at={(1.1,1.05)}, anchor=south},
				legend columns=3,
				transpose legend,
				width=\textwidth,
				]
				\addplot[domain=0.04:0.55, samples=2, thick, forget plot] {10^(\interceptkzero) * x^(\slopekzero)};
				\slopetrileft{\slopekzero}{\interceptkzero}{0.075}{0.2}{\pgfmathprintnumber{\slopekzero}}
				\addplot[thick, colork0, mark=x, mark size=3, mark options={solid}, dotted] table[x index=1, y index=3] {results_interpolation_smooth_solution_wendland_errors_k_0.txt};
				\addlegendentry{$p=0$: interpolation}
				\addplot[domain=0.04:0.55, samples=2, thick, forget plot] {10^(\interceptkone) * x^(\slopekone)};
				\slopetriright{\slopekone}{\interceptkone}{0.15}{0.275}{\pgfmathprintnumber{\slopekone}}
				\addplot[thick, colork1, mark=x, mark size=3, mark options={solid}, dotted] table[x index=1, y index=3] {results_interpolation_smooth_solution_wendland_errors_k_1.txt};
				\addlegendentry{$p=1$: interpolation}
				\addplot[domain=0.04:0.55, samples=2, thick, forget plot] {10^(\interceptktwo) * x^(\slopektwo)};
				\slopetriright{\slopektwo}{\interceptktwo}{0.075}{0.125}{\pgfmathprintnumber{\slopektwo}}
				\addplot[thick, colork2, mark=x, mark size=3, mark options={solid}, dotted] table[x index=1, y index=3] {results_interpolation_smooth_solution_wendland_errors_k_2.txt};
				\addlegendentry{$p=2$: interpolation}
				\addplot[thick, colork0, mark=*, mark size=2, mark options={solid}] table[x index=1, y index=3] {results_smooth_solution_wendland_errors_k_0.txt};
				\addlegendentry{$p=0$: deep Ritz with kernels}
				\addplot[thick, colork1, mark=*, mark size=2, mark options={solid}] table[x index=1, y index=3] {results_smooth_solution_wendland_errors_k_1.txt};
				\addlegendentry{$p=1$: deep Ritz with kernels}
				\addplot[thick, colork2, mark=*, mark size=2, mark options={solid}] table[x index=1, y index=3] {results_smooth_solution_wendland_errors_k_2.txt};
				\addlegendentry{$p=2$: deep Ritz with kernels}
			\end{axis}
		\end{tikzpicture}%
	\end{minipage}%
	\hfill%
	\begin{minipage}[b]{.48\textwidth}
		\centering%
		\begin{tikzpicture}
			\begin{axis}[
				DefaultStyle,
				log basis x={10},
				log basis y={10},
				xlabel={Mesh norm~$h$},
				xmin=0.04, xmax=0.55,
				xmode=log,
				xtick={0.5,0.25,0.1,0.05},
				xticklabels={0.5,0.25,0.1,0.05},
				ylabel={Relative $H^1$-error},
				ymin=1e-4,
				ymax=1,
				ymode=log,
				ytick={1e-08,1e-07,1e-06,1e-05,0.0001,0.001,0.01,0.1,1},
				width=\textwidth,
				]
				\addplot[domain=0.04:0.55, samples=2, thick, forget plot] {10^(\interceptkzeroh) * x^(\slopekzeroh)};
				\slopetrileft{\slopekzeroh}{\interceptkzeroh}{0.075}{0.2}{\pgfmathprintnumber{\slopekzeroh}}
				\addplot[thick, colork0, mark=x, mark size=3, mark options={solid}, dotted] table[x index=1, y index=4] {results_interpolation_smooth_solution_wendland_errors_k_0.txt};
				\addplot[thick, colork0, mark=*, mark size=2, mark options={solid}] table[x index=1, y index=4] {results_smooth_solution_wendland_errors_k_0.txt};
				\addplot[domain=0.04:0.55, samples=2, thick, forget plot] {10^(\interceptkoneh) * x^(\slopekoneh)};
				\slopetriright{\slopekoneh}{\interceptkoneh}{0.125}{0.275}{\pgfmathprintnumber{\slopekoneh}}
				\addplot[thick, colork1, mark=x, mark size=3, mark options={solid}, dotted] table[x index=1, y index=4] {results_interpolation_smooth_solution_wendland_errors_k_1.txt};
				\addplot[thick, colork1, mark=*, mark size=2, mark options={solid}] table[x index=1, y index=4] {results_smooth_solution_wendland_errors_k_1.txt};
				\addplot[domain=0.04:0.55, samples=2, thick, forget plot] {10^(\interceptktwoh) * x^(\slopektwoh)};
				\slopetriright{\slopektwoh}{\interceptktwoh}{0.05}{0.1}{\pgfmathprintnumber{\slopektwoh}}
				\addplot[thick, colork2, mark=x, mark size=3, mark options={solid}, dotted] table[x index=1, y index=4] {results_interpolation_smooth_solution_wendland_errors_k_2.txt};
				\addplot[thick, colork2, mark=*, mark size=2, mark options={solid}] table[x index=1, y index=4] {results_smooth_solution_wendland_errors_k_2.txt};
			\end{axis}
		\end{tikzpicture}%
	\end{minipage}%
	\caption{Relative errors in the~$L^2$-norm (left) and the~$H^1$-norm (right) with respect to the mesh norm~$h$ for~Wendland kernels of smoothness~$p \in \{0,1,2\}$ in the example with a smooth solution. The dashed lines indicate the error decays of the interpolation of the exact solution by the respective kernel. The estimated convergence rates for the interpolated solutions are shown as black lines.}
	\label{fig:errors-higher-regularity-wendland}
\end{figure}%
As for the Mat\'{e}rn kernels, we observe super-convergence and even faster convergence rates as suggested by the theoretical results in~\Cref{tab:wendland-kernels-regularity}.
The deep Ritz method is able to match the interpolation results asymptotically and errors are even smaller than for the interpolation.
However, we also mention the slight instabilities in the optimization for~$p=2$ when using a small mesh norm.
Similar issues were also observed for the Mat\'{e}rn kernel.
\par
Similar to the presentation in~\Cref{sec:numerical-experiments}, we also show the results when assembling and solving the linear system of equations directly.
The relative errors in this case are reported in~\Cref{fig:errors-higher-regularity-matrix-wendland}.
\begin{figure}[htbp]%
	\begin{minipage}[b]{.48\textwidth}
		\centering%
		\begin{tikzpicture}
			\begin{axis}[
				DefaultStyle,
				log basis x={10},
				log basis y={10},
				xlabel={Mesh norm~$h$},
				xmin=0.04, xmax=0.55,
				xmode=log,
				ylabel={Relative $L^2$-error},
				ymin=1e-4,
				ymax=1,
				ymode=log,
				xtick={0.5,0.25,0.1,0.05},
				xticklabels={0.5,0.25,0.1,0.05},
				ytick={1e-08,1e-07,1e-06,1e-05,0.0001,0.001,0.01,0.1,1},
				legend style={at={(1.075,1.05)}, anchor=south},
				legend columns=3,
				transpose legend,
				width=\textwidth,
				]
				\addplot[thick, colork0, mark=x, mark size=3, mark options={solid}, dotted] table[x index=1, y index=3] {results_matrix_form_smooth_solution_wendland_errors_k_0.txt};
				\addlegendentry{$p=0$: solution of linear system}
				\addplot[thick, colork1, mark=x, mark size=3, mark options={solid}, dotted] table[x index=1, y index=3] {results_matrix_form_smooth_solution_wendland_errors_k_1.txt};
				\addlegendentry{$p=1$: solution of linear system}
				\addplot[thick, colork2, mark=x, mark size=3, mark options={solid}, dotted] table[x index=1, y index=3] {results_matrix_form_smooth_solution_wendland_errors_k_2.txt};
				\addlegendentry{$p=2$: solution of linear system}
				\addplot[thick, colork0, mark=*, mark size=2, mark options={solid}] table[x index=1, y index=3] {results_smooth_solution_wendland_errors_k_0.txt};
				\addlegendentry{$p=0$: deep Ritz with kernels}
				\addplot[thick, colork1, mark=*, mark size=2, mark options={solid}] table[x index=1, y index=3] {results_smooth_solution_wendland_errors_k_1.txt};
				\addlegendentry{$p=1$: deep Ritz with kernels}
				\addplot[thick, colork2, mark=*, mark size=2, mark options={solid}] table[x index=1, y index=3] {results_smooth_solution_wendland_errors_k_2.txt};
				\addlegendentry{$p=2$: deep Ritz with kernels}
			\end{axis}
		\end{tikzpicture}%
	\end{minipage}%
	\hfill%
	\begin{minipage}[b]{.48\textwidth}
		\centering%
		\begin{tikzpicture}
			\begin{axis}[
				DefaultStyle,
				log basis x={10},
				log basis y={10},
				xlabel={Mesh norm~$h$},
				xmin=0.04, xmax=0.55,
				xmode=log,
				xtick={0.5,0.25,0.1,0.05},
				xticklabels={0.5,0.25,0.1,0.05},
				ylabel={Relative $H^1$-error},
				ymin=1e-4,
				ymax=1,
				ymode=log,
				ytick={1e-08,1e-07,1e-06,1e-05,0.0001,0.001,0.01,0.1,1},
				width=\textwidth,
				]
				\addplot[thick, colork0, mark=x, mark size=3, mark options={solid}, dotted] table[x index=1, y index=4] {results_matrix_form_smooth_solution_wendland_errors_k_0.txt};
				\addplot[thick, colork0, mark=*, mark size=2, mark options={solid}] table[x index=1, y index=4] {results_smooth_solution_wendland_errors_k_0.txt};
				\addplot[thick, colork1, mark=x, mark size=3, mark options={solid}, dotted] table[x index=1, y index=4] {results_matrix_form_smooth_solution_wendland_errors_k_1.txt};
				\addplot[thick, colork1, mark=*, mark size=2, mark options={solid}] table[x index=1, y index=4] {results_smooth_solution_wendland_errors_k_1.txt};
				\addplot[thick, colork2, mark=x, mark size=3, mark options={solid}, dotted] table[x index=1, y index=4] {results_matrix_form_smooth_solution_wendland_errors_k_2.txt};
				\addplot[thick, colork2, mark=*, mark size=2, mark options={solid}] table[x index=1, y index=4] {results_smooth_solution_wendland_errors_k_2.txt};
			\end{axis}
		\end{tikzpicture}%
	\end{minipage}%
	\caption{Relative errors in the~$L^2$-norm (left) and the~$H^1$-norm (right) with respect to the mesh norm~$h$ for~Wendland kernels of smoothness~$p \in \{0,1,2\}$ using the energy minimization problem and the linear system of equations in the example with a smooth solution.}
	\label{fig:errors-higher-regularity-matrix-wendland}
\end{figure}%
We observe that the errors stagnate for a mesh norm of roughly~$0.1$ and smaller when using the matrix-form, whereas optimizing the Dirichlet energy iteratively with random quadrature points seems to lead again to a more consistent convergence behavior.

\subsection{Singular solution on a non-convex domain}
The numerical results for Wendland kernels in the example with a singular solution from~\Cref{sec:singular-solution} are summarized in this section.
In this example, due to the lack of regularity of the solution, the convergence rates are supposed to be independent of the kernel smoothness, as already observed in~\Cref{sec:singular-solution}.
Relative errors for the interpolation of the exact solution and for the deep Ritz solution are provided in~\Cref{fig:errors-singular-wendland}.
\pgfplotstableread[skip first n=5]{results_interpolation_singular_solution_wendland_errors_k_2.txt}\tab%
\pgfplotstablecreatecol[create col/expr={ln(\thisrowno{1})/ln(10)}]{logh}\tab%
\pgfplotstablecreatecol[create col/expr={ln(\thisrowno{3})/ln(10)}]{logL2}\tab%
\pgfplotstablecreatecol[linear regression={x=logh, y=logL2}]{regr}\tab%
\xdef\slopektwo{\pgfplotstableregressiona}\xdef\interceptktwo{\pgfplotstableregressionb}%
\pgfplotstableread[skip first n=5]{results_interpolation_singular_solution_wendland_errors_k_2.txt}\tab%
\pgfplotstablecreatecol[create col/expr={ln(\thisrowno{1})/ln(10)}]{logh}\tab%
\pgfplotstablecreatecol[create col/expr={ln(\thisrowno{4})/ln(10)}]{logL2}\tab%
\pgfplotstablecreatecol[linear regression={x=logh, y=logL2}]{regr}\tab%
\xdef\slopektwoh{\pgfplotstableregressiona}\xdef\interceptktwoh{\pgfplotstableregressionb}%
\begin{figure}[htbp]%
	\begin{minipage}[b]{.48\textwidth}
		\centering%
		\begin{tikzpicture}
			\begin{axis}[
				DefaultStyle,
				log basis x={10},
				log basis y={10},
				xlabel={Mesh norm~$h$},
				xmin=0.04, xmax=0.55,
				xmode=log,
				xtick={0.5,0.25,0.1,0.05},
				xticklabels={0.5,0.25,0.1,0.05},
				ylabel={Relative $L^2$-error},
				ymin=1e-3,
				ymax=2,
				ymode=log,
				ytick={1e-08,1e-07,1e-06,1e-05,0.0001,0.001,0.01,0.1,1},
				legend style={at={(1.1,1.05)}, anchor=south},
				legend columns=3,
				transpose legend,
				width=\textwidth,
				]
				\addplot[thick, colork0, mark=x, mark size=3, mark options={solid}, dotted] table[x index=1, y index=3] {results_interpolation_singular_solution_wendland_errors_k_0.txt};
				\addlegendentry{$p=0$: interpolation}
				\addplot[thick, colork1, mark=x, mark size=3, mark options={solid}, dotted] table[x index=1, y index=3] {results_interpolation_singular_solution_wendland_errors_k_1.txt};
				\addlegendentry{$p=1$: interpolation}
				\addplot[thick, colork2, mark=x, mark size=3, mark options={solid}, dotted] table[x index=1, y index=3] {results_interpolation_singular_solution_wendland_errors_k_2.txt};
				\addlegendentry{$p=2$: interpolation}
				\addplot[thick, colork0, mark=*, mark size=2, mark options={solid}] table[x index=1, y index=3] {results_singular_solution_wendland_errors_k_0.txt};
				\addlegendentry{$p=0$: deep Ritz with kernels}
				\addplot[thick, colork1, mark=*, mark size=2, mark options={solid}] table[x index=1, y index=3] {results_singular_solution_wendland_errors_k_1.txt};
				\addlegendentry{$p=1$: deep Ritz with kernels}
				\addplot[domain=0.04:0.55, samples=2, thick, forget plot] {10^(\interceptktwo) * x^(\slopektwo)};
				\slopetrileft{\slopektwo}{\interceptktwo}{0.075}{0.2}{\pgfmathprintnumber{\slopektwo}}
				\addplot[thick, colork2, mark=*, mark size=2, mark options={solid}] table[x index=1, y index=3] {results_singular_solution_wendland_errors_k_2.txt};
				\addlegendentry{$p=2$: deep Ritz with kernels}
			\end{axis}
		\end{tikzpicture}%
	\end{minipage}%
	\hfill%
	\begin{minipage}[b]{.48\textwidth}
		\centering%
		\begin{tikzpicture}
			\begin{axis}[
				DefaultStyle,
				log basis x={10},
				log basis y={10},
				xlabel={Mesh norm~$h$},
				xmin=0.04, xmax=0.55,
				xmode=log,
				xtick={0.5,0.25,0.1,0.05},
				xticklabels={0.5,0.25,0.1,0.05},
				ylabel={Relative $H^1$-error},
				ymin=5e-2,
				ymax=2,
				ymode=log,
				ytick={0.1,0.25,0.5,1},
				yticklabels={0.1,0.25,0.5,1},
				width=\textwidth,
				]
				\addplot[thick, colork0, mark=x, mark size=3, mark options={solid}, dotted] table[x index=1, y index=4] {results_interpolation_singular_solution_wendland_errors_k_0.txt};
				\addplot[thick, colork0, mark=*, mark size=2, mark options={solid}] table[x index=1, y index=4] {results_singular_solution_wendland_errors_k_0.txt};
				\addplot[thick, colork1, mark=x, mark size=3, mark options={solid}, dotted] table[x index=1, y index=4] {results_interpolation_singular_solution_wendland_errors_k_1.txt};
				\addplot[thick, colork1, mark=*, mark size=2, mark options={solid}] table[x index=1, y index=4] {results_singular_solution_wendland_errors_k_1.txt};
				\addplot[domain=0.04:0.55, samples=2, thick, forget plot] {10^(\interceptktwoh) * x^(\slopektwoh)};
				\slopetrileft{\slopektwoh}{\interceptktwoh}{0.075}{0.2}{\pgfmathprintnumber{\slopektwoh}}
				\addplot[thick, colork2, mark=x, mark size=3, mark options={solid}, dotted] table[x index=1, y index=4] {results_interpolation_singular_solution_wendland_errors_k_2.txt};
				\addplot[thick, colork2, mark=*, mark size=2, mark options={solid}] table[x index=1, y index=4] {results_singular_solution_wendland_errors_k_2.txt};
			\end{axis}
		\end{tikzpicture}%
	\end{minipage}%
	\caption{Relative errors in the~$L^2$-norm (left) and the~$H^1$-norm (right) with respect to the mesh norm~$h$ for~Wendland kernels of smoothness~$p \in \{0,1,2\}$ in the example with a singular solution. The dashed lines indicate the error decays of the interpolation of the exact solution by the respective kernel. The estimated convergence rate for the interpolation with~$\nu=5/2$ is shown as a black line.}
	\label{fig:errors-singular-wendland}
\end{figure}%
Interestingly, we again see that the convergence rate is independent of the smoothness parameter~$p$, but at a different rate as expected from the solution regularity.
We assume that for a mesh norm of around~$h=0.05$, the asymptotic regime is not reached for Wendland kernels in this case.
However, we refrain from exploring this further in this paper.
\par
\Cref{fig:errors-singular-matrix-wendland} compares the accuracy of the deep Ritz solutions to the matrix-based ones.
\begin{figure}[htbp]%
	\begin{minipage}[b]{.48\textwidth}
		\centering%
		\begin{tikzpicture}
			\begin{axis}[
				DefaultStyle,
				log basis x={10},
				log basis y={10},
				xlabel={Mesh norm~$h$},
				xmin=0.04, xmax=0.55,
				xmode=log,
				xtick={0.5,0.25,0.1,0.05},
				xticklabels={0.5,0.25,0.1,0.05},
				ylabel={Relative $L^2$-error},
				ymin=1e-3,
				ymax=2,
				ymode=log,
				ytick={1e-08,1e-07,1e-06,1e-05,0.0001,0.001,0.01,0.1,1},
				legend style={at={(1.075,1.05)}, anchor=south},
				legend columns=3,
				transpose legend,
				width=\textwidth,
				]
				\addplot[thick, colork0, mark=x, mark size=3, mark options={solid}, dotted] table[x index=1, y index=3] {results_matrix_form_singular_solution_wendland_errors_k_0.txt};
				\addlegendentry{$p=0$: solution of linear system}
				\addplot[thick, colork1, mark=x, mark size=3, mark options={solid}, dotted] table[x index=1, y index=3] {results_matrix_form_singular_solution_wendland_errors_k_1.txt};
				\addlegendentry{$p=1$: solution of linear system}
				\addplot[thick, colork2, mark=x, mark size=3, mark options={solid}, dotted] table[x index=1, y index=3] {results_matrix_form_singular_solution_wendland_errors_k_2.txt};
				\addlegendentry{$p=2$: solution of linear system}
				\addplot[thick, colork0, mark=*, mark size=2, mark options={solid}] table[x index=1, y index=3] {results_singular_solution_wendland_errors_k_0.txt};
				\addlegendentry{$p=0$: deep Ritz with kernels}
				\addplot[thick, colork1, mark=*, mark size=2, mark options={solid}] table[x index=1, y index=3] {results_singular_solution_wendland_errors_k_1.txt};
				\addlegendentry{$p=1$: deep Ritz with kernels}
				\addplot[thick, colork2, mark=*, mark size=2, mark options={solid}] table[x index=1, y index=3] {results_singular_solution_wendland_errors_k_2.txt};
				\addlegendentry{$p=2$: deep Ritz with kernels}
			\end{axis}
		\end{tikzpicture}%
	\end{minipage}%
	\hfill%
	\begin{minipage}[b]{.48\textwidth}
		\centering%
		\begin{tikzpicture}
			\begin{axis}[
				DefaultStyle,
				log basis x={10},
				log basis y={10},
				xlabel={Mesh norm~$h$},
				xmin=0.04, xmax=0.55,
				xmode=log,
				xtick={0.5,0.25,0.1,0.05},
				xticklabels={0.5,0.25,0.1,0.05},
				ylabel={Relative $H^1$-error},
				ymin=5e-2,
				ymax=2,
				ymode=log,
				ytick={0.1,0.25,0.5,1},
				yticklabels={0.1,0.25,0.5,1},
				width=\textwidth,
				]
				\addplot[thick, colork0, mark=x, mark size=3, mark options={solid}, dotted] table[x index=1, y index=4] {results_matrix_form_singular_solution_wendland_errors_k_0.txt};
				\addplot[thick, colork0, mark=*, mark size=2, mark options={solid}] table[x index=1, y index=4] {results_singular_solution_wendland_errors_k_0.txt};
				\addplot[thick, colork1, mark=x, mark size=3, mark options={solid}, dotted] table[x index=1, y index=4] {results_matrix_form_singular_solution_wendland_errors_k_1.txt};
				\addplot[thick, colork1, mark=*, mark size=2, mark options={solid}] table[x index=1, y index=4] {results_singular_solution_wendland_errors_k_1.txt};
				\addplot[thick, colork2, mark=x, mark size=3, mark options={solid}, dotted] table[x index=1, y index=4] {results_matrix_form_singular_solution_wendland_errors_k_2.txt};
				\addplot[thick, colork2, mark=*, mark size=2, mark options={solid}] table[x index=1, y index=4] {results_singular_solution_wendland_errors_k_2.txt};
			\end{axis}
		\end{tikzpicture}%
	\end{minipage}%
	\caption{Relative errors in the~$L^2$-norm (left) and the~$H^1$-norm (right) with respect to the mesh norm~$h$ for~Wendland kernels of smoothness~$p \in \{0,1,2\}$ using the energy minimization problem and the linear system of equations in the example with a singular solution.}
	\label{fig:errors-singular-matrix-wendland}
\end{figure}%
Similar to the Mat\'{e}rn kernel results in~\Cref{sec:singular-solution}, the stiffness matrix is better conditioned for the singular examples.
The matrix-based solution achieves relative errors very close to those of the deep Ritz method.

\section{Numerical results for different neural network hyperparameters}
In this part of the appendix we provide additional numerical results using different hyperparameters for the neural networks.
Specifically, we investigate the influence of the activation function and of network architecture in terms of depth and width on the errors obtained in the deep Ritz method.
Apart from these two factors, all other settings remain as described in~\Cref{sec:numerical-experiments}.
\par
We consider six different standard activation functions used widely in the machine learning literature, see~\Cref{tab:activation functions}.
\begin{table}[htbp]
	\centering
	\begin{tabular}{*{3}{c}}
		\toprule
		Abbreviation & Full name & Function \\
		\midrule\midrule
		GELU & Gaussian Error Linear Unit & \makecell{$x\cdot\Phi(x)$ with~$\Phi$ the standard-normal\\cumulative distribution function} \\
		$\tanh$ & Hyperbolic tangent & $\frac{\exp(x)-\exp(-x)}{\exp(x)+\exp(-x)}$ \\
		ReLU & Rectified Linear Unit & $\max(0,x)$ \\
		SiLU & Sigmoid Linear Unit & \makecell{$x\cdot\sigma(x)$ with~$\sigma$ the logistic\\sigmoid~$\sigma(x)=\frac{1}{1+\exp(-x)}$} \\
		$\sin$ & Sinus & $\sin(x)$ \\
		softplus & Softplus function & $\ln(1+\exp(x))$ \\
		\bottomrule
	\end{tabular}
	\caption{Different activation functions used for the neural networks in the hyper\-para\-meter-sweeps.}
	\label{tab:activation functions}
\end{table}
\par
Moreover, we consider the following neural network architectures, where~\Cref{tab:nn-architectures} contains the number of hidden layers (depth), the number of neurons per layer (width, all layers have the same number of neurons) and the resulting number of trainable parameters.
We remark at this point that the number of layers has to be understood as the intermediate (hidden) layers.
For instance for depth~$2$, we have in total~$4$ weight matrices and bias vectors involved.
\begin{table}[htbp]
	\centering
	\begin{tabular}{c l l}
		\toprule
		Depth & Widths & Number of trainable parameters \\
		\midrule
		1 & 2, 3, 5, 8, 11, 16, 20 & 15, 25, \hphantom{0}51, 105, 177, 337, 501 \\
		2 & 1, 2, 3, 6, \hphantom{0}8, 12, 14 & \hphantom{0}9, 21, \hphantom{0}37, 109, 177, 361, 477 \\
		4 & 1, 2, 3, 4, \hphantom{0}6, \hphantom{0}8, 10 & 13, 33, \hphantom{0}61, \hphantom{0}97, 193, 321, 481 \\
		8 & 1, 2, 3, 4, \hphantom{0}5, \hphantom{0}6, \hphantom{0}7 & 21, 57, 109, 177, 261, 361, 477 \\
		\bottomrule
	\end{tabular}
	\caption{Neural network architectures with different numbers of layers and neurons used in the hyper\-para\-meter-sweeps.}
	\label{tab:nn-architectures}
\end{table}

\subsection{Diffusion equation with a solution of high regularity}
As first example, we consider again the diffusion problem from~\Cref{sec:high-regularity} where the solution is highly regular.
\subsubsection{Activation function}
The relative errors for the different activation functions are shown in~\Cref{fig:errors-higher-regularity-neural-networks-activation-functions}.
We here consider neural networks of fixed depth~$2$.
\begin{figure}[htbp]%
	\begin{minipage}[b]{.48\textwidth}
		\centering%
		\begin{tikzpicture}
			\begin{axis}[
				DefaultStyle,
				log basis x={10},
				log basis y={10},
				xlabel={Number of parameters},
				xmin=1,
				xmax=500,
				xtick={100,200,300,400,500},
				ylabel={Relative $L^2$-error},
				ymin=1e-3,
				ymax=1.5,
				ymode=log,
				ytick={1e-08,1e-07,1e-06,1e-05,0.0001,0.001,0.01,0.1,1},
				legend style={at={(1.19,1.05)}, anchor=south},
				legend columns=6,
				width=\textwidth,
				]
				\addplot[thick, viridisViolet, mark=triangle*, mark size=2, mark options={solid}] table[, x index=1, y index=2] {results_neural_network_smooth_solution_gelu_depth2_convergence_results.txt};
				\addlegendentry{GELU}
				\addplot[thick, viridisGreen, mark=*, mark size=2, mark options={solid}] table[x index=1, y index=2] {results_neural_network_smooth_solution_tanh_depth2_convergence_results.txt};
				\addlegendentry{$\tanh$}
				\addplot[thick, viridisTeal, mark=square*, mark size=2, mark options={solid}] table[x index=1, y index=2] {results_neural_network_smooth_solution_relu_depth2_convergence_results.txt};
				\addlegendentry{ReLU}
				\addplot[thick, viridisYellow, mark=x, mark size=2, mark options={solid}] table[x index=1, y index=2] {results_neural_network_smooth_solution_silu_depth2_convergence_results.txt};
				\addlegendentry{SiLU}
				\addplot[thick, matchingRed, mark=diamond*, mark size=2, mark options={solid}] table[x index=1, y index=2] {results_neural_network_smooth_solution_sin_depth2_convergence_results.txt};
				\addlegendentry{$\sin$}
				\addplot[thick, matchingOrange, mark=pentagon*, mark size=2, mark options={solid}] table[x index=1, y index=2] {results_neural_network_smooth_solution_softplus_depth2_convergence_results.txt};
				\addlegendentry{softplus}
			\end{axis}%
		\end{tikzpicture}%
	\end{minipage}%
	\hfill%
	\begin{minipage}[b]{.48\textwidth}
		\centering%
		\begin{tikzpicture}
			\begin{axis}[
				DefaultStyle,
				log basis x={10},
				log basis y={10},
				xlabel={Number of parameters},
				xmin=1,
				xmax=500,
				xtick={100,200,300,400,500},
				ylabel={Relative $H^1$-error},
				ymin=1e-3,
				ymax=1.5,
				ymode=log,
				ytick={1e-08,1e-07,1e-06,1e-05,0.0001,0.001,0.01,0.1,1},
				width=\textwidth,
				]
				\addplot[thick, viridisViolet, mark=triangle*, mark size=2, mark options={solid}] table[, x index=1, y index=3] {results_neural_network_smooth_solution_gelu_depth2_convergence_results.txt};
				\addplot[thick, viridisGreen, mark=*, mark size=2, mark options={solid}] table[x index=1, y index=3] {results_neural_network_smooth_solution_tanh_depth2_convergence_results.txt};
				\addplot[thick, viridisTeal, mark=square*, mark size=2, mark options={solid}] table[x index=1, y index=3] {results_neural_network_smooth_solution_relu_depth2_convergence_results.txt};
				\addplot[thick, viridisYellow, mark=x, mark size=2, mark options={solid}] table[x index=1, y index=3] {results_neural_network_smooth_solution_silu_depth2_convergence_results.txt};
				\addplot[thick, matchingRed, mark=diamond*, mark size=2, mark options={solid}] table[x index=1, y index=3] {results_neural_network_smooth_solution_sin_depth2_convergence_results.txt};
				\addplot[thick, matchingOrange, mark=pentagon*, mark size=2, mark options={solid}] table[x index=1, y index=3] {results_neural_network_smooth_solution_softplus_depth2_convergence_results.txt};
			\end{axis}%
		\end{tikzpicture}%
	\end{minipage}%
	\caption{Relative errors in the~$L^2$-norm (left) and the~$H^1$-norm (right) with respect to the number of parameters of fully-connected neural networks with depth~$2$ using different activation functions in the example with a smooth solution.}
	\label{fig:errors-higher-regularity-neural-networks-activation-functions}
\end{figure}%
In this example, we clearly observe that smoother activation functions lead to significantly better results, whereas~ReLU as a non-differentiable function stagnates at a large error level.
On the other hand, increasing the number of parameters does not show a large influence on the results in this case.
Already with around~$100$ parameters, the errors remain roughly constant for all of the smooth activation functions.
In this experiment, the~GELU activation function shows the most consistent performance with errors slightly smaller than for the other activation functions.
This motivates our choice of~GELU as primary example in~\Cref{sec:numerical-experiments}.

\subsubsection{Numbers of layers and neurons}
The results for different neural network architectures are presented in~\Cref{fig:errors-higher-regularity-neural-networks-depth} for the example with a smooth solution.
We here stick to~GELU as activation function which turned out to be the most reliable activation function in this example, see the previous results in~\Cref{fig:errors-higher-regularity-neural-networks-activation-functions}.
\begin{figure}[htbp]%
	\begin{minipage}[b]{.48\textwidth}
		\centering%
		\begin{tikzpicture}
			\begin{axis}[
				DefaultStyle,
				log basis x={10},
				log basis y={10},
				xlabel={Number of parameters},
				xmin=1,
				xmax=510,
				xtick={100,200,300,400,500},
				ylabel={Relative $L^2$-error},
				ymin=1e-3,
				ymax=1,
				ymode=log,
				ytick={1e-08,1e-07,1e-06,1e-05,0.0001,0.001,0.01,0.1,1},
				legend style={at={(1.19,1.05)}, anchor=south},
				legend columns=6,
				width=\textwidth,
				]
				\addplot[thick, viridisGreen, mark=*, mark size=2, mark options={solid}] table[, x index=1, y index=2] {results_neural_network_smooth_solution_gelu_depth1_convergence_results.txt};
				\addlegendentry{Depth $1$}
				\addplot[thick, viridisViolet, mark=triangle*, mark size=2, mark options={solid}] table[x index=1, y index=2] {results_neural_network_smooth_solution_gelu_depth2_convergence_results.txt};
				\addlegendentry{Depth $2$}
				\addplot[thick, matchingRed, mark=square*, mark size=2, mark options={solid}] table[x index=1, y index=2] {results_neural_network_smooth_solution_gelu_depth4_convergence_results.txt};
				\addlegendentry{Depth $4$}
				\addplot[thick, viridisYellow, mark=x, mark size=2, mark options={solid}] table[x index=1, y index=2] {results_neural_network_smooth_solution_gelu_depth8_convergence_results.txt};
				\addlegendentry{Depth $8$}
			\end{axis}%
		\end{tikzpicture}%
	\end{minipage}%
	\hfill%
	\begin{minipage}[b]{.48\textwidth}
		\centering%
		\begin{tikzpicture}
			\begin{axis}[
				DefaultStyle,
				log basis x={10},
				log basis y={10},
				xlabel={Number of parameters},
				xmin=1,
				xmax=510,
				xtick={100,200,300,400,500},
				ylabel={Relative $H^1$-error},
				ymin=1e-3,
				ymax=1,
				ymode=log,
				ytick={1e-08,1e-07,1e-06,1e-05,0.0001,0.001,0.01,0.1,1},
				width=\textwidth,
				]
				\addplot[thick, viridisGreen, mark=*, mark size=2, mark options={solid}] table[, x index=1, y index=3] {results_neural_network_smooth_solution_gelu_depth1_convergence_results.txt};
				\addplot[thick, viridisViolet, mark=triangle*, mark size=2, mark options={solid}] table[x index=1, y index=3] {results_neural_network_smooth_solution_gelu_depth2_convergence_results.txt};
				\addplot[thick, matchingRed, mark=square*, mark size=2, mark options={solid}] table[x index=1, y index=3] {results_neural_network_smooth_solution_gelu_depth4_convergence_results.txt};
				\addplot[thick, viridisYellow, mark=x, mark size=2, mark options={solid}] table[x index=1, y index=3] {results_neural_network_smooth_solution_gelu_depth8_convergence_results.txt};
			\end{axis}%
		\end{tikzpicture}%
	\end{minipage}%
	\caption{Relative errors in the~$L^2$-norm (left) and the~$H^1$-norm (right) with respect to the number of parameters of fully-connected neural networks with different depths in the example with a smooth solution. For every depth, the widths are adjusted in such a way that the parameter count remains comparable.}
	\label{fig:errors-higher-regularity-neural-networks-depth}
\end{figure}%
We observe that the errors are close to independent of the depth of the neural networks.
Only the depth~$4$ and depth~$8$ networks seem to be slightly less stable in their training.
We therefore suggest to use neural networks with depth up to~$2$ which do not require special treatment to obtain satisfactory results.

\subsection{Singular solution on a non-convex domain}
Finally, we also discuss the results of the hyperparameter sweep for the neural networks in the case of the example from~\Cref{sec:singular-solution}.
\subsubsection{Activation function}
\Cref{fig:errors-singular-neural-networks-activation-functions} shows the relative errors for different activation functions and fixed depth~$2$.
\begin{figure}[htbp]%
	\begin{minipage}[b]{.48\textwidth}
		\centering%
		\begin{tikzpicture}
			\begin{axis}[
				DefaultStyle,
				log basis x={10},
				log basis y={10},
				xlabel={Number of parameters},
				xmin=1,
				xmax=500,
				xtick={100,200,300,400,500},
				ylabel={Relative $L^2$-error},
				ymin=1e-3,
				ymax=0.25,
				ymode=log,
				ytick={1e-08,1e-07,1e-06,1e-05,0.0001,0.001,0.01,0.1,1},
				legend style={at={(1.19,1.05)}, anchor=south},
				legend columns=6,
				width=\textwidth,
				]
				\addplot[thick, viridisViolet, mark=triangle*, mark size=2, mark options={solid}] table[, x index=1, y index=2] {results_neural_network_singular_solution_gelu_depth2_convergence_results.txt};
				\addlegendentry{GELU}
				\addplot[thick, viridisGreen, mark=*, mark size=2, mark options={solid}] table[x index=1, y index=2] {results_neural_network_singular_solution_tanh_depth2_convergence_results.txt};
				\addlegendentry{$\tanh$}
				\addplot[thick, viridisTeal, mark=square*, mark size=2, mark options={solid}] table[x index=1, y index=2] {results_neural_network_singular_solution_relu_depth2_convergence_results.txt};
				\addlegendentry{ReLU}
				\addplot[thick, viridisYellow, mark=x, mark size=2, mark options={solid}] table[x index=1, y index=2] {results_neural_network_singular_solution_silu_depth2_convergence_results.txt};
				\addlegendentry{SiLU}
				\addplot[thick, matchingRed, mark=diamond*, mark size=2, mark options={solid}] table[x index=1, y index=2] {results_neural_network_singular_solution_sin_depth2_convergence_results.txt};
				\addlegendentry{$\sin$}
				\addplot[thick, matchingOrange, mark=pentagon*, mark size=2, mark options={solid}] table[x index=1, y index=2] {results_neural_network_singular_solution_softplus_depth2_convergence_results.txt};
				\addlegendentry{softplus}
			\end{axis}%
		\end{tikzpicture}%
	\end{minipage}%
	\hfill%
	\begin{minipage}[b]{.48\textwidth}
		\centering%
		\begin{tikzpicture}
			\begin{axis}[
				DefaultStyle,
				log basis x={10},
				log basis y={10},
				xlabel={Number of parameters},
				xmin=1,
				xmax=500,
				xtick={100,200,300,400,500},
				ylabel={Relative $H^1$-error},
				ymin=5e-2,
				ymax=0.5,
				ymode=log,
				ytick={0.1,0.25},
				yticklabels={0.1,0.25},
				legend style={at={(1.1,1.05)}, anchor=south},
				width=\textwidth,
				]
				\addplot[thick, viridisViolet, mark=triangle*, mark size=2, mark options={solid}] table[, x index=1, y index=3] {results_neural_network_singular_solution_gelu_depth2_convergence_results.txt};
				\addplot[thick, viridisGreen, mark=*, mark size=2, mark options={solid}] table[x index=1, y index=3] {results_neural_network_singular_solution_tanh_depth2_convergence_results.txt};
				\addplot[thick, viridisTeal, mark=square*, mark size=2, mark options={solid}] table[x index=1, y index=3] {results_neural_network_singular_solution_relu_depth2_convergence_results.txt};
				\addplot[thick, viridisYellow, mark=x, mark size=2, mark options={solid}] table[x index=1, y index=3] {results_neural_network_singular_solution_silu_depth2_convergence_results.txt};
				\addplot[thick, matchingRed, mark=diamond*, mark size=2, mark options={solid}] table[x index=1, y index=3] {results_neural_network_singular_solution_sin_depth2_convergence_results.txt};
				\addplot[thick, matchingOrange, mark=pentagon*, mark size=2, mark options={solid}] table[x index=1, y index=3] {results_neural_network_singular_solution_softplus_depth2_convergence_results.txt};
			\end{axis}%
		\end{tikzpicture}%
	\end{minipage}%
	\caption{Relative errors in the~$L^2$-norm (left) and the~$H^1$-norm (right) with respect to the number of parameters of fully-connected neural networks with depth~$2$ using different activation functions in the example with a singular solution.}
	\label{fig:errors-singular-neural-networks-activation-functions}
\end{figure}%
In this example, the differences are again relatively small between the activation functions.
Surprisingly, also for this less regular solution, ReLU yields the least accurate approximations independent of the number of parameters.
We suppose that this is related to the relatively large part of the domain where the solution is smooth, see~\Cref{fig:solution-plot-singular}, which renders regular activation functions slightly better suited.
Both, GELU and~$\tanh$ perform quite similar and we decided, for consistency with the first example, to use~GELU as activation function for all further tests.

\subsubsection{Numbers of layers and neurons}
For the different architectures, \Cref{fig:errors-singular-neural-networks-depth} summarizes the results in the example with a singular solution on a non-convex domain using again~GELU as activation function.
\begin{figure}[htbp]%
	\begin{minipage}[b]{.48\textwidth}
		\centering%
		\begin{tikzpicture}
			\begin{axis}[
				DefaultStyle,
				log basis x={10},
				log basis y={10},
				xlabel={Number of parameters},
				xmin=1,
				xmax=510,
				xtick={100,200,300,400,500},
				ylabel={Relative $L^2$-error},
				ymin=1e-3,
				ymax=0.25,
				ymode=log,
				ytick={1e-08,1e-07,1e-06,1e-05,0.0001,0.001,0.01,0.1,1},
				legend style={at={(1.19,1.05)}, anchor=south},
				legend columns=6,
				width=\textwidth,
				]
				\addplot[thick, viridisGreen, mark=*, mark size=2, mark options={solid}] table[, x index=1, y index=2] {results_neural_network_singular_solution_gelu_depth1_convergence_results.txt};
				\addlegendentry{Depth $1$}
				\addplot[thick, viridisViolet, mark=triangle*, mark size=2, mark options={solid}] table[x index=1, y index=2] {results_neural_network_singular_solution_gelu_depth2_convergence_results.txt};
				\addlegendentry{Depth $2$}
				\addplot[thick, matchingRed, mark=square*, mark size=2, mark options={solid}] table[x index=1, y index=2] {results_neural_network_singular_solution_gelu_depth4_convergence_results.txt};
				\addlegendentry{Depth $4$}
				\addplot[thick, viridisYellow, mark=x, mark size=2, mark options={solid}] table[x index=1, y index=2] {results_neural_network_singular_solution_gelu_depth8_convergence_results.txt};
				\addlegendentry{Depth $8$}
			\end{axis}%
		\end{tikzpicture}%
	\end{minipage}%
	\hfill%
	\begin{minipage}[b]{.48\textwidth}
		\centering%
		\begin{tikzpicture}
			\begin{axis}[
				DefaultStyle,
				log basis x={10},
				log basis y={10},
				xlabel={Number of parameters},
				xmin=1,
				xmax=510,
				xtick={100,200,300,400,500},
				ylabel={Relative $H^1$-error},
				ymin=5e-2,
				ymax=0.5,
				ymode=log,
				ytick={0.1,0.25},
				yticklabels={0.1,0.25},
				width=\textwidth,
				]
				\addplot[thick, viridisGreen, mark=*, mark size=2, mark options={solid}] table[, x index=1, y index=3] {results_neural_network_singular_solution_gelu_depth1_convergence_results.txt};
				\addplot[thick, viridisViolet, mark=triangle*, mark size=2, mark options={solid}] table[x index=1, y index=3] {results_neural_network_singular_solution_gelu_depth2_convergence_results.txt};
				\addplot[thick, matchingRed, mark=square*, mark size=2, mark options={solid}] table[x index=1, y index=3] {results_neural_network_singular_solution_gelu_depth4_convergence_results.txt};
				\addplot[thick, viridisYellow, mark=x, mark size=2, mark options={solid}] table[x index=1, y index=3] {results_neural_network_singular_solution_gelu_depth8_convergence_results.txt};
			\end{axis}%
		\end{tikzpicture}%
	\end{minipage}%
	\caption{Relative errors in the~$L^2$-norm (left) and the~$H^1$-norm (right) with respect to the number of parameters of fully-connected neural networks with different depths in the example with a singular solution. For every depth, the widths are adjusted in such a way that the parameter count remains comparable.}
	\label{fig:errors-singular-neural-networks-depth}
\end{figure}%
As we already observed in~\Cref{fig:errors-higher-regularity-neural-networks-depth}, using a depth of~$8$ leads to optimization issues and instabilities.
In this case, we make the same observation also for depth~$4$.
The results for depth~$1$ and~$2$ are quite similar with slightly smaller errors obtained when employing a depth of~$2$.

\bibliographystyle{plainnat}
\bibliography{references}

\end{document}